\documentclass[final]{siamltex}
\usepackage{amsfonts, amsmath, amssymb}
\usepackage{graphicx}
\usepackage{tabularx}
\usepackage{array}
\usepackage{mathrsfs}
\usepackage{bm}
\usepackage{epsfig}
\usepackage{cite}
\usepackage{color}
\usepackage{multirow}
\newtheorem{example}[theorem]{Example}
\newtheorem{remark}[theorem]{Remark}

\def\ppp{{p}}

\title{A multilevel stochastic collocation method\\ for partial differential equations with random input data}

\author{
A.~L.~Teckentrup\thanks{Department of Scientific Computing,
Florida State University,
400 Dirac Science Library,  
Tallahassee, FL 32306-4120
({\tt ateckentrup@fsu.edu}).}
\and
P.~Jantsch\thanks{Department of Mathematics,
University of Tennessee,
1403 Circle Drive, 
Knoxville, TN 37916
({\tt pjantsch@utk.edu}).}
\and
C.~G.~Webster\thanks{Department of Computational and Applied Mathematics,
Oak Ridge National Laboratory, One Bethel Valley Road,
P.O.~Box 2008, MS-6164, Oak Ridge, TN 37831-6164
({\tt webstercg@ornl.gov}).}
\and
M.~Gunzburger\thanks{Department of Scientific Computing,
Florida State University,
400 Dirac Science Library,  
Tallahassee, FL 32306-4120
({\tt gunzburg@fsu.edu}).}
}

\begin{document}
\maketitle

\newcommand{\slugmaster}{%
\slugger{}{}{}{}{}}

\renewcommand{\thefootnote}{\arabic{footnote}}

\begin{abstract}
Stochastic collocation methods for approximating the solution of partial differential equations with random input data (e.g., coefficients and forcing terms) suffer from the curse of dimensionality whereby increases in the stochastic dimension cause an explosion of the computational effort. We propose and analyze a multilevel version of the stochastic collocation method that, as is the case for multilevel Monte Carlo (MLMC) methods, uses hierarchies of spatial approximations to reduce the overall computational complexity. In addition, our proposed approach utilizes, for approximation in stochastic space, a sequence of multi-dimensional interpolants of increasing fidelity  which can then be used for approximating statistics of the solution as well as for building high-order surrogates featuring faster convergence rates. A rigorous convergence and computational cost analysis of the new multilevel stochastic collocation method is provided, demonstrating its advantages compared to standard single-level stochastic collocation approximations as well as MLMC methods. Numerical results are provided that illustrate the theory and the effectiveness of the new multilevel method.
\end{abstract}



\section{Introduction}

Nowadays, mathematical modeling and computer simulations are used extensively in many scientific and engineering fields, usually with the goal of understanding or predicting the behavior of a system given its inputs such as the computational domain, model parameter values, and source terms. However, whether stemming from incomplete or inaccurate knowledge or from some inherent variability in the system, often these inputs may be subject to uncertainty. In order to correctly predict the behavior of the system, it is especially pertinent to understand and propagate the effect of the input uncertainty to the output of the simulation, i.e., to the solution of the mathematical model.

In this paper, we consider systems which are modeled by elliptic partial differential equations (PDEs) with random input data. We work under the \emph{finite-dimensional noise assumption}, i.e., we assume that the random inputs are characterized by a finite-dimensional random vector. When enough information is available to completely characterize the randomness in the inputs, probability theory provides a natural setting for quantifying uncertainties. The object of our computations is the accurate calculation of solution of stochastic elliptic PDEs or statistics of some functional of the solution of the PDE. For instance, in addition to the solution itself, one might be interested in the expected value or variance of the solution in a given region of the computational domain.

A large number of methods have been developed for the numerical solution of PDEs with random inputs; see, e.g., \cite{GWZ14} and the references cited therein. The most popular approach is the Monte Carlo (MC) methods which involves random sampling of the input vector of random variables (also referred to as the \emph{stochastic parameter space}) and the solution of the deterministic PDE at each of the sample points. In addition to the benefits of simple implementation and a natural decoupling of the stochastic and spatial degrees of freedom, MC methods feature a convergence rate that is independent of the dimension of the stochastic space. This makes it particularly attractive for high-dimensional problems. However, the convergence is in general very slow and, especially in case the stochastic space is only of moderate dimension and the solution of the PDE or of a functional of interest is smooth, better convergence rates can be achieved using more sophisticated methods.

Stochastic collocation (SC) methods~\cite{Babuska:2007,Nobile:2008_1,Nobile:2008_2} are similar to MC methods in the sense that they involve only the solution of a sequence of deterministic PDEs at given sample points in the stochastic space. However, rather than randomly chosen samples, SC methods use a deterministic grid of points at which one then solves the corresponding deterministic PDE, and then builds an interpolant, either using global Lagrange-type polynomials \cite{Babuska:2007,Nobile:2008_1,Nobile:2008_2} or even local hierarchical basis functions \cite{GWZ13,Ma:2009p4298}. For problems where the solution is a smooth function of the random input variables and the dimension of the stochastic space is moderate, SC methods have been shown to converge much faster than MC methods. 

Unfortunately, for most problems, stochastic collocation methods suffer from the  \emph{curse of dimensionality}, a phrase that refers to the deterioration of the convergence rate and the explosion of computational effort as the dimension of the stochastic space increases.   
In this paper, we introduce a multilevel stochastic collocation (MLSC) approach for reducing the computational cost incurred by standard, i.e., single level, SC methods. Drawing inspiration from multigrid solvers for linear equations, the main idea behind multilevel methods is to utilize a hierarchical sequence of spatial approximations to the underlying PDE model that are then combined with stochastic discretizations in such a way as to minimize computational cost. Starting with the pioneering works \cite{Heinrich:2001} in the field of integral equations and  \cite{Giles:2008} in the field of computational finance, the multilevel approach has been successfully applied to many applications of MC methods; see, e.g., \cite{Barth:2011gz,Charrier:2013jv,dh11,gkss11,gr12,hsst12}. The MLSC methods we consider in this paper are similar to the multilevel quadrature schemes studied in \cite{Harbrecht:2013}. However, our focus is on the analysis of the computational complexity of the multilevel algorithms and also includes results for functionals of the solution. In particular, we prove new interpolation error bounds on functionals of the solution that are needed for the analysis of the MLSC methods. 

The outline of the paper is as follows. In Section \ref{sec:problemsetting}, we introduce the mathematical problem, the main notation used throughout, the assumptions on the
parametrization of the random inputs that are used to transform the original stochastic problem into a deterministic parametric version, and needed assumptions about the regularity of the solution of the PDE. A description of the spatial and stochastic approximations as well as the formulation of the MLSC method follows in Section \ref{sec:form}. In Section \ref{sec:analysis}, we provide a general convergence and complexity analysis for the MLSC method. As an example of a specific single level SC approach satisfying our interpolation assumptions, we describe, in Section \ref{sec:SC}, a generalized sparse grid stochastic collocation approach based on global Lagrange interpolation. In Section \ref{sec:num}, we provide numerical results that illustrate the theoretical results and complexity estimates and also explore issues related to the implementation of the MLSC method.

\section{Problem Setting}\label{sec:problemsetting}

Consider the problem of approximating the solution of an elliptic partial differential equation (PDE) with random input data. To this end, let $D \subset \mathbb{R}^d,$ $d=1,2,3$, denote a bounded, Lipschitz domain with boundary denoted by $\partial D$ and let $(\Omega,\mathscr{F},\mathbb{P})$ denote a complete probability space. Here, $\Omega$ denotes the set of outcomes, $\mathscr{F} \subset 2^{\Omega}$ the $\sigma$-algebra of events, and $\mathbb{P}:\mathscr{F} \rightarrow [0,1]$ a complete probability measure. Given random fields $a(\omega,\mathbf{x}), f(\omega,\mathbf{x}) : \Omega\times \overline{D} \rightarrow \mathbb{R}$, the model problem we consider is stated as follows: find $u(\omega,\mathbf{x}) : \Omega\times \overline{D} \rightarrow \mathbb{R}$ such that almost surely
\begin{equation}
\label{eq:mod}
\left\{
\begin{array}{rll}
-\nabla(a(\omega,\mathbf{x}) \cdot \nabla u(\omega,\mathbf{x})) &= f(\omega,\mathbf{x}) &\mbox{in } D\\
					u(\omega,\mathbf{x}) &= 0  &\mbox{on } \partial D.
\end{array}
\right.
\end{equation}

We make the following assumptions on $a$ and $f$:
\begin{itemize}
\item[{\bf {\bf A1}.}] (\emph{Finite-dimensional noise}) The random fields $a$ and $f$ are determined by a finite number $N$ of random variables, denoted by the random vector
$\bm{y}(\omega) := [ y_1(\omega), \ldots y_N(\omega) ] : \Omega \rightarrow \mathbb{R}^N$.  

\item[{\bf {\bf A2}.}] (\emph{Boundedness}) The image $\Gamma_n := y_n(\Omega)$ of $y_n$ is bounded for all $n \in \{1,\dots, N\}$ and, with $\Gamma = \prod_{n=1}^{N} \Gamma_n$, the random variables $\bm y$ have a joint probability density function $\rho({\bm y})= \prod_{n=1}^{N}\widetilde\rho(y_n)\in L^{\infty}(\Gamma)$, where $\widetilde\rho(\cdot):[-1,1]\to\mathbb{R}$ denotes the one-dimensional PDF corresponding to the probability space of the random fields. Without loss of generality, we assume that $\Gamma = [-1,1]^N$.
\end{itemize}

\begin{remark}
\ Another setting having a finite number of random variables is that of the coefficient $a$ and the forcing function $f$ depending on a a finite number of independent scalar random physical parameters, e.g., diffusivities, reaction rates, porosities, elastic moduli, etc. In this case, each of the $N$ parameters would have its own PDF $\rho_n(y_n)$, $n=1,\ldots,N$, so that the joint PDF is now given by $\rho({\bm y})= \prod_{n=1}^{N}\rho_n(y_n)$. The algorithms discussed in this work all apply equally well to this setting.
\end{remark}

Under assumptions {\bf A1} and {\bf A2}, it follows from the Doob-Dynkin Lemma that the solution $u$ to \eqref{eq:mod} can also be characterized in terms of the random vector $\bm y(\omega)$. The solution $u(\omega,\mathbf{x})$ thus has a deterministic, parametric equivalent $u(\bm y,\mathbf{x})$, with the probability space $(\Gamma, \mathcal B, \rho(\bm y) \mathrm{d} \bm y)$ taking the place of $(\Omega,\mathscr{F},\mathbb{P})$; see, e.g., \cite{Babuska:2007}. Here, $\mathcal B$ denotes the Borel $\sigma$-algebra generated by the open subsets of $\Gamma$. In what follows, we will therefore denote the solution by $u(\bm y,\mathbf{x})$ for $\bm y \in \Gamma$ and $\mathbf{x} \in D$. Then we also assume:
\begin{itemize}
\item[{\bf {\bf A3}.}] (\emph{Existence and uniqueness}) The coefficients $a(\omega, \mathbf{x})$ are uniformly bounded and coercive, i.e., there exists $a_{min}>0 \text{ and } a_{max}<\infty \text{ such that } $ 
$$
\mbox{Prob}\left[\omega\in\Omega:\,\, 
a_{min}\leq a(\bm{y}(\omega),\mathbf{x})\leq a_{max}\,\,\,\forall\, \mathbf{x}\in\overline{D}\right]=1
$$
and $f\in L^2_{\rho}(\Gamma; H^{-1}(D))$ so that the problem \eqref{eq:mod} admits a unique solution $u \in L^2_\rho(\Gamma; H^1_0(D))$ with realizations in $H_0^1(D)$, i.e., $u(\bm{y}(\omega),\cdot)\in H^1_0(D)$ almost surely.
\end{itemize}

\noindent Here, given a Banach space $X(D)$ of functions on $D$, the weighted Bochner spaces $L^q_\rho(\Gamma; X(D))$ for $1 \leq q < \infty$ are defined by
$$
	L_{\rho}^q(\Gamma;X(D)) = \Big\{ v:\Gamma \rightarrow X(D)\ | \ v \text{ is measurable and } \int_{\Gamma} \|v(\bm y,\cdot)\|^q_{X(D)} \rho(\bm y) \mathrm{d}\bm y < \infty \Big\}
$$
with corresponding norm $\|\cdot\|_{L_{\rho}^q(\Gamma;X(D))}$ given by
$$
\|v\|^q_{L_{\rho}^q(\Gamma;X(D))} = \int_{\Gamma} \|v(\bm y,\cdot)\|^q_{X(D)} \rho(\bm y) \mathrm{d}\bm y.
$$

Assumption {\bf A1} is naturally satisfied by random fields that only depend on a finite set of parameters, e.g.,
\begin{equation*}
a(\omega, \mathbf{x})= a(\bm{y}(\omega), \mathbf{x}) =  a_0 + \sum_{n=1}^N y_n(\omega) a_n(\mathbf{x}),
\quad \{a_n\}_{n=1}^N\in L^2(D),
\end{equation*}
and similarly for $f(\omega, \mathbf{x}) \approx f(\bm{y}(\omega),\mathbf{x})$, where $\bm{y}(\omega)$ is a vector of independent random variables.  If this is not the case, approximations of $a$ and $f$ that satisfy assumption {\bf A1} can be obtained by appropriately truncating a spectral expansion such as the Karhunen-Lo\`eve expansion  \cite{Ghanem_Spanos_1991}. This introduces an additional error; see \cite{Nobile:2008_1} for a discussion of the effect of this error on the convergence of stochastic collocation methods and \cite{fst05, charrier12} for bounds on the truncation error.

Assumption {\bf A2} can be weakened to include the case of unbounded random variables such as Gaussian variables. See \cite{Babuska:2007} for an analysis of the interpolation error and note that, with only minor modifications, the multilevel stochastic collocation method introduced in this paper also applies to unbounded random variables. Furthermore, assumption {\bf A3} can be weakened to include coefficients $a$ that are not uniformly coercive; see \cite{Charrier:2013jv, tsgu13}. 

Finally, we remark that the multilevel stochastic collocation method proposed in this paper is not specific to the model problem \eqref{eq:mod}; it can be applied also to higher-order PDEs and other types of boundary conditions.

\section{Hierarchical multilevel stochastic collocation methods}
\label{sec:form}
We begin by recalling that standard stochastic collocation (SC) methods generally build an approximation of the solution $u$ by evaluating a spatial approximation $u_h(\bm{y}, \cdot)\in V_h$ at a given set of points $\{{\bm y}_m\}_{m=1}^{M}$ in $\Gamma$, where $V_h\subset H_0^1(D)$ is a finite-dimensional subspace. In other words, we compute $\{u_h(\bm{y}_m, \cdot)\}_{m=1}^{M}$.  Then, given a basis 
$\left\{\phi_m({\bm y})\right\}_{m=1}^{M}$ for the space $\mathcal{P}_M=\mbox{span}\left\{\phi_m({\bm y})\right\}_{m=1}^{M} \subset L^2_{\rho}(\Gamma)$, we use those samples to construct the fully discrete approximation given by the interpolant
\begin{equation}\label{def:stdsc}
u^\mathrm{(SL)}_{M,h}(\bm y, \mathbf x)=
\mathcal{I}_M[u_h](\bm y, \mathbf{x}) =  \sum_{m=1}^{M} c_m(\mathbf{x}) \, \phi_m({\bm y}),
\end{equation}
where the coefficients $c_m(\mathbf x)$ are determined by the interpolating conditions $\mathcal{I}_M[u_h](\bm y_m, \mathbf{x})  = u_h(\bm y_m,\mathbf x)$ for $m=1, \dots, M$. In \eqref{def:stdsc}, we label the standard SC approximation by `SL' to indicate that that approximation is constructed using a single set of points $\{{\bm y}_m\}_{m=1}^{M}$ in stochastic space, in contrast to the multilevel approximations considered below that use a hierarchy of point sets; thus, we refer to \eqref{def:stdsc} as a {\em single level} approximation.
A wide range of choices for the interpolation points $\{{\bm y}_m\}_{m=1}^{M}$ and basis functions $\{\phi_m({\bm y})\}_{m=1}^{M}$ are possible. A particular example
of the approximation \eqref{def:stdsc}, namely global Lagrange interpolation on generalized sparse grids, is given in Section \ref{sec:SC}.

Convergence of the SC approximation \eqref{def:stdsc} is often assessed in the natural $L_{\rho}^2(\Gamma;H^1_0(D))$-norm, and the goal is to determine a bound on the error 
$\|u - \mathcal{I}_M[u_h]|_{L_{\rho}^2(\Gamma;H^1_0(D))}$. To obtain a good approximation with SC methods, it is necessary in general to use accurate spatial approximations $u_h$ and a large number $M$ of collocation points. To determine the coefficients $c_m(\mathbf{x})$ of the interpolant \eqref{def:stdsc}, the method requires the computation of $u_h(\bm y_m, \cdot )$ for $m=1,\ldots,M$ so that, in practice, the cost can grow quickly with increasing $N$. Therefore, to reduce the overall cost, we consider a multilevel version of SC methods that combines different levels of fidelity of both the spatial and parameter approximations. 

\subsection{Spatial approximation} \label{ssec:appr}

For spatial approximation, we use hierarchical family of finite element discretizations \cite{brenner_scott, ciarlet}. As discussed in \cite{Harbrecht:2013}, the formulation of the multilevel method does not depend on the specific spatial discretization scheme used and the results readily hold for other choices. For $k \in \mathbb N_0$, define a hierarchy of nested finite element spaces
\begin{equation*}
	V_{h_0} \subset V_{h_1} \subset \dots \subset V_{h_k} \subset \dots \subset H^1_0(D),
\end{equation*}
where each $V_{h_k}$ consists of continuous, piecewise polynomial functions on a shape regular triangulation $\tau_{h_k}$ of $D$ having maximum mesh spacing parameter $h_k$. Note that $k$ merely serves to index the given spaces; the approximation properties of the space $V_{h_k}$ is governed by $h_k$. For simplicity, we assume that the triangulations $\{\tau_{h_k}\}_{k \in \mathbb N_0}$ are generated by iterative uniform subdivisions of the initial triangulation $\tau_0$; this implies that $h_k = \eta^{-k}h_0$ for some $\eta \in \mathbb N$, $\eta>1$ and that indeed the corresponding finite element spaces are nested. 

\begin{remark}\ For simplicity, we have assumed that the finite element family of spaces is nested, and in fact, are constructed by a series of uniform subdivisions of a parent mesh with mesh size $h_0$. Neither of these assumptions are necessary for our algorithms or conclusions to hold, provided $\eta_1 \leq h_k/h_{k+1} \leq \eta_2$ for some $0 < \eta_1 < \eta_2 < \infty$ and all $k \in \mathbb N_0$; in such cases, the finite element spaces are not necessarily nested.
\end{remark}

We also let $u_{h_k}(\bm y,\cdot)$ denote the Galerkin projection of $u(\bm y,\cdot)$ onto $V_{h_k}$, i.e., $u_{h_k}\in V_{h_k}$ denotes the finite element approximation. Note that $u_{h_k}({\bm y},\cdot)$ is still a function on the stochastic parameter space $\Gamma$. We assume the following approximation property of the finite element spaces $\{V_{h_k}\}_{k \in \mathbb N_0}$:

\begin{itemize}
\item[{\bf {\bf A4}.}] There exist positive constants $\alpha$ and $C_s$, independent of $h_k$, such that for all $k \in \mathbb N_0$,
$$ 
\|u-u_{h_k}\|_{L_{\rho}^2(\Gamma;H^1_0(D))} \leq C_s \, h_k^{\alpha}.
$$
\end{itemize}
In general, the rate $\alpha$ depends on the (spatial) regularity of $u$, which in turn depends on the regularity of $a$ and $f$ as well as on the geometry of the domain $D$. For example, if $a$, $f$, and $D$ are sufficiently regular so that $u \in L^2_{\rho}(\Gamma; H^2(D))$, assumption {\bf A4} holds with $\alpha =1$ and $C_s$ dependent only on $a$ and $\|u\|_{L^2_{\rho}(\Gamma; H^2(D))}$. For additional examples and detailed analyses of finite element errors, see \cite{tsgu13}.

\subsection{Stochastic interpolation} \label{ssec:sto_int}

For stochastic approximation, we use interpolation over $\Gamma$, where we assume $u \in C^0(\Gamma;H^1_0(D))$. The specific choice of interpolation scheme is not crucial at this juncture. We begin by letting $\left\{ \mathcal{I}_{M_k} \right\}_{k=0}^{\infty}$ denote a sequence of interpolation operators $\mathcal I_{M_k} : C^0(\Gamma) \rightarrow L^2_\rho(\Gamma)$ using $M_k$ points. We assume the following:

\begin{itemize}
\item[{\bf {\bf A5}.}] There exist positive constants $C_I, C_\zeta$, $\beta$, and $\mu$, and a Banach space $\Lambda(\Gamma;H^1_0(D)) \subset L^2_\rho(\Gamma;H^1_0(D))$ containing the finite element approximations $\{u_{h_k}\}_{k \in \mathbb N_0}$ such that for all $v \in \Lambda(\Gamma;H^1_0(D))$ and all $k \in \mathbb N_0$
$$
\|v-\mathcal{I}_{M_k} v\|_{L_{\rho}^2(\Gamma;H^1_0(D))} \leq \, C_I \, \sigma_k \, \zeta(v)
$$
for some function $\zeta : \Lambda(\Gamma;H^1_0(D)) \rightarrow \mathbb R$ and a decreasing sequence $\sigma_k$ that admit the estimates 
$$
\zeta(u_{h_k}) \leq C_\zeta \, h_0^\beta \quad \text{ and } \quad \zeta(u_{h_{k+1}} - u_{h_{k}}) \leq C_\zeta \, h_{k+1}^\beta.
$$
\end{itemize}

\begin{remark}\ 
As in the previous section, $k$ is merely an index; we use the same index for the hierarchies of spatial and stochastic approximations because, in the multilevel SC method we introduce below, these two hierarchies are closely connected. 
\end{remark}

\begin{remark}\ 
$\sigma_k$ determines the approximation properties of the interpolant. Moreover, we allow non-unique interpolation operators in the sequence, i.e., it is possible that, for any $k =0,\ldots, \infty$, $M_{k+1}=M_k$ and therefore $\mathcal{I}_{M_{k+1}}=\mathcal{I}_{M_k}$ and $\sigma_{k+1} = \sigma_k $. Thus, although the spatial approximation improves with increasing $k$, i.e., $h_{k+1}<h_k$, we allow for the parameter space approximation for the index $k+1$ remaining the same as that for $k$. 
\end{remark}

In Section \ref{sec:SC}, assumption {\bf A5} is shown to hold, with $\sigma_k = M_k^{-\mu}$, for global Lagrange interpolation using generalized sparse grids. The bounds on the function $\zeta$ in assumption {\bf A5} are shown to be the key to balancing spatial and stochastic discretizations through the multilevel formulation.  Crucially, we make use of the fact that the interpolation error is proportional to the size of the function being interpolated, measured in an appropriate norm. In the case of the model problem \eqref{eq:mod}, this norm is usually related to the (spatial) $H^1_0(D)$-norm. The bounds in assumption {\bf A5} then arise from the fact that for any $k \in N_0$, $\|u_{h_k}\|_{H^1_0(D)}$ is bounded by a constant, independent of $k$, whereas $\|u_{h_k}-u_{h_{k-1}}\|_{H^1_0(D)}$ decays with $h_k^\beta$ for some $\beta>0$. We usually have $\beta = \alpha$, where $\alpha$ is as in assumption {\bf A4}. Note that we have chosen to scale the bound on $\zeta(u_{h_k})$ by $h_0^\beta$ to simplify calculations. Because $h_0$ is a constant, this does not affect the nature of the assumption.

\subsection{Formulation of the multilevel method} \label{ssec:formulation}

As in the previous sections, denote by $\{u_{h_k}\}_{k \in \mathbb N_0}$ and $\{\mathcal I_{M_k}\}_{k \in \mathbb N_0}$ sequences of spatial approximations and interpolation operators in parameter space, respectively. Then, for any $K \in \mathbb N$, the formulation of the multilevel method begins with the simple telescoping identity
\begin{equation} \label{eqn:MLform}
	u_{h_K} = \sum_{k=0}^{K} (u_{h_k} - u_{h_{k-1}}),
\end{equation}
where, for simplicity, we set $u_{h_{-1}}:=0$.

It follows from assumption {\bf A5} that as $k \rightarrow \infty$, less accurate interpolation operators are needed in order to estimate $u_{h_k} - u_{h_{k-1}}$ to achieve a required accuracy. We therefore define our multilevel interpolation approximation as
	\begin{equation}\label{def:ml_app}
		u_K^\mathrm{(ML)} := \sum_{k=0}^{K} \mathcal{I}_{M_{K-k}}[u_{h_k} - u_{h_{k-1}}] \\
		= \sum_{k=0}^{K}\Big( u^\mathrm{(SL)}_{M_{K-k},h_k} - u^\mathrm{(SL)}_{M_{K-k},h_{k-1}}\Big). 
	\end{equation}
Rather than simply interpolating $u_{h_K}$, this approximation uses different levels of interpolation on each difference $u_{h_k} - u_{h_{k-1}}$ of finite element approximations. To preserve convergence, the estimator uses the most accurate interpolation operator $\mathcal I_{M_K}$ on the coarsest spatial approximation $u_{h_0}$ and the least accurate interpolation operator $\mathcal I_{M_0}$ on the finest spatial approximation $u_{h_K} - u_{h_{K-1}}$. Note that in \eqref{def:ml_app} a single index $k$ is used to select appropriate spatial and stochastic approximations and thus these approximations are indeed closely related.

\section{Analysis of the multilevel approximation} \label{sec:analysis}

This section is devoted to proving the convergence of the multilevel approximation defined in Section \ref{ssec:formulation} and analyzing its computational complexity. We first prove, in Section \ref{ssec:convergence}, a general error bound, whereas in Sections \ref{ssec:comp_analysis} and \ref{ssec:comp_analysis_func} we prove a bound on the computational complexity in the particular case of an algebraic decay of the interpolation errors.

\subsection{Convergence analysis} \label{ssec:convergence}

We consider the convergence of the multilevel approximation $u_K^\mathrm{(ML)}$ to the true solution $u$ in the natural norm $\|\cdot\|_{L_{\rho}^q(\Gamma;H_0^1(D))}$.

First, we use the triangle inequality to split the error into the sum of a spatial discretization error and a stochastic interpolation error, i.e., 
\begin{equation}\label{eq:err_split}
\|u-u_K^\mathrm{(ML)}\|_{L_{\rho}^2(\Gamma;H^1_0(D))} \leq {\underbrace{\|u - u_{h_K}\|}_{(I)}} { }_{L_{\rho}^2(\Gamma;H^1_0(D))} + \underbrace{\|u_{h_K} - u_K^\mathrm{(ML)}\|}_{(II)} { }_{L_{\rho}^2(\Gamma;H^1_0(D))}.
\end{equation}
The aim is to prove that with the interpolation operators $\{\mathcal I_{M_k}\}_{k=0}^K$ chosen appropriately, the stochastic interpolation error ($II$) of the multilevel approximation converges at the same rate as the spatial discretization error ($I$), hence resulting in a convergence result for the total error.

For the spatial discretization error ($I$), it follows immediately from assumption {\bf A4} that
$$
	(I) \leq C_s h_K^{\alpha}.
$$

From \eqref{eqn:MLform}, the triangle inequality, and assumption {\bf A5}, we estimate the stochastic interpolation error:
$$
\begin{aligned}
	(II) &= \Big\|\sum_{k=0}^{K} (u_{h_k} - u_{h_{k-1}}) -  \mathcal{I}_{M_{K-k}}(u_{h_k} - u_{h_{k-1}})\Big\|_{L_{\rho}^2(\Gamma;H^1_0(D))}\\
	&\leq \sum_{k=0}^{K} \big\|(u_{h_k} - u_{h_{k-1}}) -  \mathcal{I}_{M_{K-k}}(u_{h_k} - u_{h_{k-1}})\big\|_{L_{\rho}^2(\Gamma;H^1_0(D))}\\
	&\leq \sum_{k=0}^{K} C_I \, C_\zeta \, \sigma_{K-k} \, h_k^\beta.
\end{aligned}
$$
To obtain an error of the same size as $(I)$, we choose interpolation operators such that 
\begin{equation} \label{eqn:epsk}
	\sigma_{K-k} \leq C_s \, \big((K+1)\, C_I \, C_\zeta\big)^{-1} \, h_K^\alpha \, h_k^{-\beta}.
\end{equation}
Continuing from above, 
$$
 (II) \leq \sum_{k=0}^{K} C_s \, \big((K+1)\, C_I \, C_\zeta)\big)^{-1} \, h_K^\alpha \, h_k^{-\beta} C_I \, C_\zeta \, h_k^\beta
	= C_s h_K^\alpha, 
$$
as required. It follows that with $\sigma_k$ as in \eqref{eqn:epsk}  
$$
\|u-u_K^\mathrm{(ML)}\|_{L_{\rho}^2(\Gamma;H^1_0(D))} \leq 2 \, C_s \, h_K^\alpha.
$$

\subsection{Cost analysis} \label{ssec:comp_analysis}
{
We now proceed to analyze the computational cost of the MLSC method. We consider the {\em $\varepsilon$-cost of the estimator}, denoted here by $C_\varepsilon^\mathrm{ML}$, which is the computational cost required to achieve a desired accuracy $\varepsilon$. In order to quantify this cost, we use the convergence rates of the spatial discretization error and, for the stochastic interpolation error, the rates given by assumptions {\bf A4} and {\bf A5}. In particular, we will assume that {\bf A5} holds with $\sigma_k = M_k^{-\mu}$ for some $\mu > 0$.

\begin{remark}\
The choice $\sigma_k = M_k^{-\mu}$ best reflects approximations based on SC methods that employ sparse grids. In particular, as mentioned in Section \ref{ssec:sto_int}, algebraic decay holds for the generalized sparse grid interpolation operators considered in Section~\ref{sec:SC}; see Theorem~\ref{theorem:SCerror}. For other possible choices in the context of quadrature, see \cite{Harbrecht:2013}.
\end{remark}

In general, the MLSC method involves solving, for each $k$, the deterministic PDE for each of the $M_k$ sample points from $\Gamma$; in fact, according to \eqref{def:ml_app}, two solves are needed,  one for each of two spatial grid levels. Thus, we also require a bound on the cost, which we denote by $C_k$, of computing $u_{h_k} - u_{h_{k-1}}$ at a sample point. We assume:
\begin{itemize}
\item[{\bf {\bf A6}.}] There exist positive constants $\gamma$ and $C_c$, independent of $h_k$, such that $C_k \leq C_c \, h_k^{-\gamma}$ for all $k \in \mathbb N_0$.
\end{itemize} 

\noindent If an optimal linear solver is used to solve the finite element equations for $u_{h_k}$, this assumption holds with $\gamma \approx d$ (see, e.g., \cite{brenner_scott}), where $d$ is the spatial dimension. Note that the constant $C_c$ will in general depend on the refinement ratio $\eta$ described in Section 
\ref{ssec:appr}.

We quantify the total computational cost of the MLSC approximation \eqref{def:ml_app} using the metric
\begin{equation}\label{eq:cost} 
C^\mathrm{(ML)} = \sum_{k=0}^K M_{K-k} \, C_k.
\end{equation}

We now have the following result for the $\varepsilon$-cost of the MLSC method required to achieve an accuracy $\|u- u_K^\mathrm{(ML)}\|_{L_{\rho}^2(\Gamma;H^1_0(D))} \leq \varepsilon$.

\begin{theorem}
\label{theorem:main}
Suppose assumptions {\bf A4}--{\bf A6} hold with $\sigma_k = M_k^{-\mu}$, and assume that $\alpha \geq \min(\beta,\mu \gamma)$. Then, for any $\varepsilon < \exp[-1]$, there exists an integer $K$ such that
$$\|u-u_K^\mathrm{(ML)}\|_{L_{\rho}^2(\Gamma;H^1_0(D))} \leq \varepsilon $$
and
\begin{equation}
C_\varepsilon^\mathrm{(ML)} \lesssim 
	\begin{cases}
		\varepsilon^{-\frac{1}{\mu}}, & \text{if } \beta > \mu \gamma \\[1ex]
		\varepsilon^{-\frac{1}{\mu}}|\log \varepsilon |^{1 + \frac{1}{\mu}}  & \text{if } \beta = \mu \gamma \\[1ex]
		\varepsilon^{-\frac{1}{\mu}-\frac{\gamma \mu - \beta}{\alpha \mu}}  & \text{if } \beta < \mu \gamma. 
	\end{cases}
\end{equation}
\end{theorem}

\begin{proof}
As in \eqref{eq:err_split}, we consider separately the two error contributions $(I)$ and $(II)$. To achieve the desired accuracy, it is sufficient to bound both error contributions by $\frac{\varepsilon}{2}$. Without loss of generality, for the remainder of this proof we assume $h_0=1$. If this is not the case, we simply need to rescale the constants $C_s$, $C_\zeta$, and $C_c$.

First, we choose $K$ large enough so that $(I) \leq \frac{\varepsilon}{2}$. By assumption {\bf A4}, it is sufficient to require $ C_s h_K^{\alpha} \leq \frac{\varepsilon}{2}$. Because the hierarchy of meshes $\{h_k\}_{k \in \mathbb N_0}$ is obtained by uniform refinement, $h_k = \eta^{-k} h_0 = \eta^{-k}$, and we have
\begin{equation}\label{def:K}
 	h_K \leq \big(\frac{\varepsilon}{2 C_s}\big)^{1/\alpha} \quad\text{if}\quad K = \left\lceil \frac{1}{\alpha} \log_\eta \big( \frac{2 C_s}{\varepsilon}\big)\right\rceil.
 \end{equation}
This fixes the total number of levels $K$. 

In order to obtain the multilevel estimator with the smallest computational cost, we now determine the $\{M_k\}_{k=0}^K$ so that the computational cost \eqref{eq:cost} is minimized, subject to the requirement $ (II)  \leq \frac{\varepsilon}{2}$. Treating the $M_k$ as continuous variables, the Lagrange multiplier method, together with assumptions {\bf A4} and {\bf A5}, results in the optimal choice 
\begin{equation}\label{eq:opt_eta}
 M_{K-k} 
= \big(2 \, C_I \, C_\zeta \, \mathcal{S}(\eta,K)\big)^{1/\mu} \, \varepsilon^{-1/\mu} \, \eta^{-\frac{k(\beta + \gamma)}{\mu + 1}}, 
\end{equation}
where
$$
	\mathcal{S}(\eta,K) = \sum_{k=0}^{K} \eta^{-k(\frac{\beta - \gamma\mu}{\mu+1})}.
$$ 
Because $M_{K-k}$ given by \eqref{eq:opt_eta} is, in general, not an integer, we choose
\begin{equation}\label{eq:opt_etar}
 M_{K-k} 
= \left\lceil\left(2 \, C_I \, C_\zeta \, \mathcal{S}(\eta,K)\right)^{1/\mu} \, \varepsilon^{-1/\mu} \, \eta^{-\frac{k(\beta + \gamma)}{\mu + 1}}\right\rceil. 
\end{equation}
Note that this choice determines the sequence $\{M_k\}_{k=0}^K$ and consequently $\{\mathcal I_{M_k}\}_{k=0}^K$. Also note that, in practice, this choice may not be possible for all interpolation schemes; see Remark \ref{rem:lev}. 

With the number of samples $M_{K-k}$ fixed, we now examine the complexity of the multilevel approximation:
\begin{align*}
C_{\varepsilon}^\mathrm{(ML)} & = \sum_{k=0}^{K} M_{K-k} C_k \eqsim \sum_{k=0}^{K} M_{K-k} \,  \eta^{k \gamma} \\
		& \lesssim \sum_{k=0}^{K}   \Big(\frac{\varepsilon}{\mathcal{S}(\eta,K)}\Big)^{-\frac{1}{\mu}} \,\eta^{-k\frac{\beta + \gamma}{\mu + 1}} \, \eta^{k\gamma} +  \sum_{k=0}^{K} \eta^{k\gamma}\\
		& \eqsim \varepsilon^{-\frac{1}{\mu}} \mathcal{S}(\eta,K)^{\frac{1}{\mu}} \sum_{k=0}^{K}    \eta^{-k\frac{\beta + \gamma -\gamma(\mu + 1)}{\mu + 1}} +  \sum_{k=0}^{K} \eta^{k\gamma}\\
		& \eqsim \varepsilon^{-\frac{1}{\mu}} \mathcal{S}(\eta,K)^{\frac{1}{\mu}} \sum_{k=0}^{K}    \eta^{-k\frac{\beta -\gamma\mu}{\mu + 1}} +  \sum_{k=0}^{K} \eta^{k\gamma}\\
		& \eqsim \varepsilon^{-\frac{1}{\mu}} \mathcal{S}(\eta,K)^{1 + \frac{1}{\mu}} +  \sum_{k=0}^{K} \eta^{k\gamma}. \\
\end{align*}
To bound the cost in terms of $\varepsilon$, first note that because $K < \frac{1}{\alpha} \log_\eta ( 2 C_s/\varepsilon) + 1$ by \eqref{def:K}, we have 
\begin{equation}\label{eq:sum_cost}
\sum_{k=0}^{K} \eta^{k\gamma} \leq \frac{\eta^{\gamma L}}{1-\eta^{-\gamma}} \leq \frac{\eta^\gamma (2 C_s)^{\gamma/\alpha}}{1-\eta^{-\gamma}} \varepsilon^{-\gamma/\alpha}.
\end{equation}

Next, we need to consider different values of $\beta$ and $\mu$. When $\beta > \gamma\mu $, $\mathcal{S}(\eta,K)$ is a geometric sum that converges to a limit independent of $K$. Because $\alpha \geq \gamma\mu $ implies that $\varepsilon^{-\gamma/\alpha} \leq \varepsilon^{-\frac{1}{\mu}}$ for $\varepsilon < \exp[-1]$, we have $C_{\varepsilon}^\mathrm{(ML)} \lesssim \varepsilon^{-\frac{1}{\mu}}$ in this case. 

When $\beta = \gamma\mu$, we find that $\mathcal{S}(\eta,K) = K+1$, and so, using \eqref{def:K} and $\alpha \geq \mu \gamma$,
$$ C_{\varepsilon}^\mathrm{(ML)} \lesssim \varepsilon^{-\frac{1}{\mu}} (K+1)^{1 + \frac{1}{\mu}} + \varepsilon^{-\frac{\gamma}{\alpha}} \simeq \varepsilon^{-\frac{1}{\mu}} |\log \varepsilon|^{1 + \frac{1}{\mu}}. $$

For the final case of $\beta < \gamma\mu$, we reverse the index in the sum $\mathcal{S}(\eta,K)$ to obtain a geometric sequence
\begin{equation*}
	 \mathcal{S}(\eta,K) = \sum_{k=0}^{K}    \eta^{(k-K)\frac{\beta - \gamma\mu}{\mu + 1}} 
			= \eta^{-K\frac{\beta - \gamma\mu}{\mu + 1}} \sum_{k=0}^{K}    \eta^{-k(\frac{\gamma\mu - \beta}{\mu + 1})} 
			\lesssim \varepsilon^{\frac{\beta-\gamma\mu}{\alpha(\mu+1)}}.
\end{equation*}
Because $\alpha \geq \beta$, this gives 
$$ C_{\varepsilon}^\mathrm{(ML)}\lesssim \varepsilon^{-\frac{1}{\mu}} \varepsilon^{\frac{\beta-\gamma\mu}{\alpha(\mu+1)} (1 + \frac{1}{\mu})}  + \varepsilon^{-\frac{\gamma}{\alpha}} \eqsim \varepsilon^{-\frac{1}{\mu}-\frac{\gamma \mu- \beta}{\alpha \mu}}.$$
This completes the proof.
\end{proof}
}

\begin{remark}\label{rem:lev}\ {\bf Error and quadrature level.} In this section, we characterized the convergence of the interpolation errors in terms of the number of interpolation points $M_k$. Yet when computing quadratures based on sparse grid techniques (see Section \ref{sec:SC}), an arbitrary number of points will not in general have an associated sparse grid. Thus, choosing an interpolant using the optimal number of points according to \eqref{eq:opt_etar} may not be possible in practice. However, in light of estimates such as \cite[Lemma 3.9] {Nobile:2008_1}, it is not unreasonable to make the assumption that given any number of points $M$, there exists an interpolant using $\widetilde{M}$ points, with
\begin{equation} \label{eqn:Mbeta}
	M \leq \widetilde{M} \leq C M^{\delta}
\end{equation}
for some $\delta \geq 1$. We can think of $\delta$ as measuring the inefficiency of our sparse grids in representing higher-dimensional polynomial spaces. Using \eqref{eqn:Mbeta},
one can proceed as in Theorem \ref{theorem:main} to derive a bound on the $\varepsilon$-cost of the resulting multilevel approximation.

Another possibility would be to solve a discrete, constrained minimization problem to find optimal interpolation levels, relying on convergence results for the interpolation error in terms of the interpolation level rather than number of points; see \cite[Theorem 3.4]{Nobile:2008_2}. However, our cost metric relies on precise knowledge of the number of points, making theoretical comparison difficult.
\end{remark}

\begin{remark}\label{rem:canc}\ {\bf Cancellations and computational cost.}
The cost estimate \eqref{eq:cost} takes into consideration the cost of all the terms in the multilevel estimator \eqref{def:ml_app}. However, when the same interpolation operator is used on two consecutive levels, terms in the multilevel approximation cancel and need in fact not be computed. For example, if $\mathcal{I}_{M_{K-k}} = \mathcal{I}_{M_{K-k-1}}$, then 
$$
 \mathcal{I}_{M_{K-k}} (u_{h_k} - u_{h_{k-1}}) + \mathcal{I}_{M_{K-k-1}} (u_{h_{k+1}} - u_{h_k}) = \mathcal{I}_{M_{K-k}} (u_{h_{k+1}} - u_{h_{k-1}})
$$
so that the computation of the interpolants of $u_{h_k}$ is not necessary. Especially in the context of sparse grid interpolation, in practice we choose the same interpolation grid for several consecutive levels, leading to a significant reduction in the actual computational cost compared to that estimated in Theorem \ref{theorem:main}. The effect of these cancellations is clearly visible in some of the numerical experiments of Section \ref{sec:num}.
\end{remark}

\subsubsection{Comparison to single level collocation methods}\label{ssec:cost_comp}
Under the same assumptions as in Theorem \ref{theorem:main}, for any $M_{sl} \in \mathbb{N}_0$ and $h_{sl}$, the error in the standard single-level SC approximation \eqref{def:stdsc} can be bounded by
$$\|u - u_{M_{sl},h_{sl}}^\mathrm{(SL)}\|_{L_{\rho}^2(\Gamma;H^1_0(D))} \leq C_s \, h_{sl}^{\alpha} + C_I \, \zeta(u_h) \, M_{sl}^{-\mu} . $$
To make both contributions equal to $\varepsilon/2$, it suffices to choose $h_{sl} \eqsim \varepsilon^{1/\alpha}$ and $M_{sl} \eqsim \varepsilon^{-1/\mu}$. This choice determines $M_{sl}$ and hence $\mathcal{I}_{M_{sl}}$. The computational cost to achieve a total error of $\varepsilon$ is then bounded by
$$C_{\varepsilon}^\mathrm{(SL)} \eqsim  h^{-\gamma} M_{sl} \eqsim \varepsilon^{- \frac{1}{\mu} -\frac{\gamma}{\alpha}}.$$
A comparison with the bounds on computational complexity proved in Theorem \ref{theorem:main} shows clearly the superiority of the multilevel method. 

In the case $\beta > \gamma \mu$, the cost of obtaining one sample of $u_{h_k}$ grows slowly with respect to $k$, and most of the computational effort of the multilevel approximation is at the coarsest level $k=0$. The savings in cost compared to single level SC hence correspond to the difference in cost between obtaining samples $u_{h_0}$ on the coarse grid $h_0$ and obtaining samples $u_{h_K}$ on the fine grid $h=h_K$ used by the single-level method. This gives a saving of $(h/h_0)^\gamma \eqsim \varepsilon^{\gamma/\alpha}$.

The case $\beta = \mu \gamma$ corresponds to the computational effort being spread evenly across the levels, and, up to a log factor, the savings in cost are again of order $ \varepsilon^{\gamma/\alpha}$.

In contrast, when $\beta < \gamma \mu$, the computational cost of computing one sample of $u_{h_k}$ grows quickly with respect to $k$, and most of the computational effort of the multilevel approximation is on the finest level $k=K$. The benefits compared to single level SC hence corresponds approximately to the difference between $M_K$ and $M_{sl}$. This gives a savings of $M_K/M_{sl} \eqsim (h_K^\beta)^{1/\mu} \eqsim \varepsilon^{\beta/\alpha \mu}$.

\subsection{Multilevel approximation of functionals}\label{ssec:comp_analysis_func}

In applications, it is often of interest to bound the error in the expected value of a functional $\psi$ of the solution $u$, where $\psi:H^1_0(D) \rightarrow \mathbb R$. Similar to \eqref{def:stdsc}, the SC approximation of $\psi(u)$ is given by
\begin{equation}\label{def:stdscccc}
\psi^\mathrm{(SL)}_{k,h}[u] = \mathcal{I}_{M_k}\left[ \psi(u_h) \right ] 
\end{equation}
and, similar to \eqref{def:ml_app}, the multilevel interpolation approximation of $\psi(u)$ is given by
\begin{equation}\label{def:ml_appccc}
	\psi^\mathrm{(ML)}_K[u]   := \sum_{k=0}^{K} \mathcal{I}_{M_{K-k}}\big(\psi(u_{h_k}) - \psi(u_{h_{k-1}})\big),
\end{equation}
where, as before, we set $u_{h_{-1}}:=0$ and we also assume, without loss of generality, that $\psi(0)=0$. 	 
Note that in the particular case of linear functionals $\psi$, we in fact have \[ \psi_{k,h}^\mathrm{(SL)}[u] = \psi (u_{k,h}^\mathrm{(SL)})\quad \text{and}\quad \psi^\mathrm{(ML)}_K[u] = \psi(u_K^\mathrm{(ML)}).\]

Analogous to Theorem \ref{theorem:main}, we have the following result about the $\varepsilon$-cost for the error $\big|\mathbb E\big[\psi(u)-\psi^\mathrm{(ML)}_K[u]\big]\big|$ in the expected value of the multilevel approximation of functionals.

\begin{proposition}
\label{thm:comp_func}
Suppose there exist positive constants $\alpha, \beta, \mu, \gamma, C_s, C_I, C_\zeta, C_c$ and a real-valued function $\zeta$ such that $\alpha \geq \min(\beta,\mu \gamma)$ and, for all $k \in \mathbb N_0$, assume that
\vspace{1ex}
\begin{itemize}
\item[\bf{{\bf F1}.}] $|\mathbb E[\psi(u) - \psi(u_{h_k})]| \leq C_s \, h_k^{\alpha}$ \vspace{1ex}
\item[\bf{{\bf F2}.}] $\big|\mathbb E\big[\psi(u_{h_k})-\psi(u_{h_{k-1}}) - \mathcal I_{M_{K-k}}(\psi(u_{h_k})-\psi(u_{h_{k-1}}))\big]\big| \leq C_I \, M_{K-k}^{-\mu} \, \zeta(\psi(u_{h_k})-\psi(u_{h_{k-1}}))$ 
\item[\bf{{\bf F3}.}] $\zeta(\psi(u_{h_k})-\psi(u_{h_{k-1}})) \leq C_\zeta \, h_k^\beta$
\item[\bf{{\bf F4}.}] $C_k = C_c \, h_k^{-\gamma}$.
\end{itemize}
\vspace{1ex}
Then, for any $\varepsilon < \exp[-1]$,  there exists a value $K$ such that
$$\big|\mathbb E\big[\psi(u)-\psi^\mathrm{(ML)}_K(u)\big]\big| \leq \varepsilon, $$
with computational cost $C_\varepsilon^\mathrm{(ML)}$ bounded as in Theorem \ref{theorem:main}.
\end{proposition}

The assumptions {\bf {\bf F1}}--{\bf {\bf F4}} are essentially the same as the assumptions {\bf A4}--{\bf A6} of Theorem \ref{theorem:main}, with perhaps different values for the constants $C_s$, $C_I$, $C_\zeta$, and $C_c$. Certainly, bounded linear functionals have this inheritance property. In Section \ref{sec:SC}, we give some examples of nonlinear functionals that also have this property.

\section{Global sparse grid interpolation} \label{sec:SC}

In this section, we provide a specific example of a single level SC approach, given by 
\eqref{def:stdsc}, that will be used to construct the interpolation operators in our MLSC 
approach.  As such, we briefly recall generalized multi-dimensional (possibly sparse grid) interpolation, as well as theoretical results related to the interpolation operator.  For a more thorough description, see 
\cite{Babuska:2007,Nobile:2008_1,Nobile:2008_2, Beck:2011p5113}.

\begin{remark}\
In this section, we again introduce a second notion of {\em levels}. The levels here should not be confused with the levels used previously. For the latter, `levels' refer to members of hierarchies of spatial and stochastic approximations, both of which were indexed by $k$. In this section, `levels' refer to a sequence, indexed by $l$, of stochastic polynomial spaces and corresponding point sets used to construct a specific sparse grid interpolant. The result of this construction, i.e., of using the levels indexed by $l$, is the interpolants used in the previous sections that were indexed by $k$.
\end{remark} 

\subsection{Construction and convergence analysis of multi-dimensional interpolants} \label{ssec:sgconst}

The construction of the interpolant in the $N$-dimensional space $\Gamma = \prod_{n=1}^{N} \Gamma_n$ is based on sequences of one-dimensional Lagrange interpolation operators $\{\mathcal U_n^{\ppp(l)}\}_{l \in \mathbb N} : C^0(\Gamma_n) \rightarrow \mathcal{P}_{\ppp(l)-1}(\Gamma_n)$, where $\mathcal{P}_\ppp(\Gamma_n)$ denotes the space of polynomials of degree $\ppp$ on $\Gamma_n$. In particular, for each $n=1,\ldots, N$, let $l\in\mathbb{N}_+$ denote the one-dimensional level of approximation and let $\{y_{n,j}^{(l)}\}_{j=1}^{\ppp({l}) }\subset\Gamma_n$ denote a sequence of one-dimensional interpolation points in $\Gamma_n$.  Here, 
$\ppp({l}):\mathbb{N}_+\rightarrow\mathbb{N}_+$ is such that $\ppp(1) = 1$ and $\ppp({l}) < \ppp({l+1})$ for $l= 2,3,\ldots$, so that $\ppp(l)$ strictly increases with $l$ and defines the total number of collocation points at level $l$. For a univariate function $v\in C^0(\Gamma_n)$,
we define $\mathcal{U}_n^{\ppp({l}) }$ by 
\begin{equation}\label{tttttt}
\mathcal{U}_n^{\ppp({l})}[v](y_n) = \sum_{j=1}^{\ppp({l}) } v\big(y_{n,j}^{(l)}\big)\varphi_{n,j}^{(l)}(y_n)
\quad\text{for } l_n=1,2,\ldots,
\end{equation}
where $\varphi_{n,j}^{(l)}\in\mathcal{P}_{\ppp({l})  - 1}(\Gamma_n)$, $j=1,\ldots,\ppp({l}) $, are Lagrange fundamental polynomials of degree $\ppp({l})  -1$, which are completely determined by the property $\varphi_{n,j}^{(l)}( y_{n,i}^{(l)} ) = \delta_{i,j}$.

Using the convention that $\mathcal{U}_n^{p(0)} = 0$, we introduce the { difference operator} given by
\begin{equation}
\label{eq:Delta}
\Delta_n^{\ppp({l})} = \mathcal{U}^{\ppp({l})}_n - \mathcal{U}^{\ppp(l-1)}_n.
\end{equation}

For the multivariate case, we let $\bm{l}=(l_1, \ldots, l_N)\in\mathbb{N}_+^N$ denote a 
multi-index and $L\in\mathbb{N}_+$ denote the  total level of the sparse grid approximation.  
Now, from \eqref{eq:Delta}, the 
{\em $L$-th level generalized sparse-grid approximation} of $v \in C^0(\Gamma)$ is given by
\begin{equation}
\label{eq:SG}
\mathcal{A}^{p, g}_{L} [ v ] =
\sum_{g(\bm{l})\leq {L}}\bigotimes_{n=1}^N \Delta_n^{\ppp({l_n})} [ v ],
\end{equation}
where $g:\mathbb{N}^N_+\rightarrow\mathbb{N}$ is another strictly increasing function 
that defines the mapping between the multi-index $\bm{l}$ and the 
level $L$ used to construct the sparse grid. 
The single level approximation \eqref{eq:SG} requires the 
independent evaluation of $v$
on a deterministic set of {\em distinct collocation points} given by
\begin{equation*}
\mathcal{H}^{p,g}_L = \bigcup_{g(\bm{l})\leq {L}} \bigotimes_{n=1}^N
\left\{y_{n,j}^{(l_n)}\right\}_{j=1}^{\ppp({l_n})}
\end{equation*}
having cardinality $M_L$.

\begin{remark}\
For the MLSC method, the interpolation operators ${\mathcal I}_{M_k}$ introduced in Section \ref{ssec:sto_int} are chosen as $\mathcal{A}^{p, g}_{L}$ with $M_k=M_L$. Indeed, this is done for the numerical examples of Section \ref{sec:num}.
\end{remark}

The particular choices of the one-dimensional growth rate $\ppp(l)$ and the function $g(\bm l)$ define
a general multi-index set $\mathcal{J}^{p,g}(L)$ used in the construction of the sparse grid, and the corresponding underlying polynomial space of the approximation denoted $\mathcal{P}_{\mathcal{J}^{p,g}(L)}(\Gamma)$ \cite{Beck:2011p5113,GWZ14}.
Some examples of functions $\ppp(l)$ and $g(\bm l)$ and the corresponding polynomial approximation spaces  are given in Table \ref{sgex}. In the last example in the table $\bm{\alpha}=(\alpha_1,\ldots, \alpha_N) \in \mathbb{R}_{+}^N$ is a vector of weights reflecting the anisotropy of the system, i.e., the relative importance of each dimension \cite{Nobile:2008_2}; we then define 
$\displaystyle\alpha_{{min}} := \min_{n=1,\ldots,N} \alpha_n$. 
The corresponding anisotropic versions of the other approximations and corresponding polynomial subspaces can be analogously constructed.

\begin{table}[h!]
\caption{The functions 
$p:\mathbb{N}_+\rightarrow\mathbb{N}_+$
and 
$g:\mathbb{N}^N_+\rightarrow\mathbb{N}$
and the corresponding polynomial subspaces.
}
\label{sgex}
\begin{center}
\renewcommand*\arraystretch{1.5}
\begin{tabular}{l c c c}
\bf Polynomial Space & & $\ppp(l)$ & $g(\bm l)$ \\
\hline \vspace{.5ex}
	\bf Tensor product&& $\ppp(l)=l$  & 
	$\displaystyle\max_{1\leq n\leq N} (l_n -1)$ \\ 
	\bf Total degree && $\ppp(l)=l$  & 
	$\sum_{n=1}^N (l_n - 1)$ \\ 
	\bf Hyperbolic cross && $\ppp(l)=l$ & $\prod_{n=1}^N ( l_n  - 1 )$ \\
	\bf Sparse Smolyak && $\ppp(l)=2^{l-1} + 1,\, l>1$ & $ \sum_{n=1}^N (l_n - 1)$ \\ \vspace{.5ex}
	\bf Anisotropic Sparse Smolyak && $\ppp(l)=2^{l-1} + 1,\,l>1$ & $\sum_{n=1}^N \frac{\alpha_n}{\alpha_{{min}}}(l_n - 1)$, $\bm{\alpha} \in \mathbb{R}_{+}^N$  \\
	\hline
\end{tabular}
\end{center}
\end{table}
 
Table \ref{sgex} defines several polynomial spaces. A means for constructing a basis for polynomial subspaces consists of the selecting a set of points and the defining basis functions based on those points, e.g., Lagrange fundamental polynomials. For Smolyak polynomial spaces, the most popular choice of points are the sparse grids based on the one-dimensional Clenshaw-Curtis abscissas \cite{clenshaw_curtis_60} which are the extrema of Chebyshev polynomials, including the end-point extrema. The resulting multi-dimensional points are given by, for level $l$ and in the particular case $\Gamma_n = [-1,1]$ and $\ppp(l)> 1$,
$$
 y^{(l)}_{n,j} = -\cos\left({\pi (j-1)\over \ppp(l)  -1}\right) \quad \mbox{for $j=1,\dots, \ppp(l)$}.
$$
In particular, the choice $\ppp(l)$ given in Table \ref{sgex} for the Smolyak case results in a {\em nested} family of one-dimensional abscissas, i.e., 
$\big\{y^{(l)}_{n,j}\big\}_{j=1}^{\ppp(l) }\subset \big\{y_{n,j}^{(l+1)}\big\}_{j=1}^{\ppp(l+1) }$,
so that the sparse grids are also nested, i.e.,
${\mathcal{H}}^{p,g}_L\subset {\mathcal{H}}^{p,g}_{L+1}$. Using $g(\bm l)$ in \eqref{eq:SG}, 
given as in Table \ref{sgex} for the Smolyak polynomial space, corresponds
to the most widely used sparse-grid approximation, as first described in \cite{Smolyak_63}.

Other nested families of sparse grids can be constructed from, e.g., the Newton-Cotes and Gauss-Patterson one-dimensional abscissas.

\begin{remark}\
In general, the growth rate $\ppp(l)$ can be chosen as any increasing function on $\mathbb N$.  However,  
to construct the approximation \eqref{eq:SG} in the tensor product, total degree, hyperbolic cross, and Smolyak polynomial spaces, the required functions $p$ and $g$ are described in Table \ref{sgex}.  Moreover, if the underlying abscissas can be nested, as for the Clenshaw-Curtis points described above, the approximation 
\eqref{eq:SG} remains a Lagrange interpolant. 
For non-nested point families, such as standard Gaussian abscissas, 
the approximation \eqref{eq:SG} is no longer guaranteed to be an interpolant. 
As such, additional errors may be introduced that must be analyzed  
\cite{Nobile:2008_1}.  
\end{remark}

\subsubsection{General convergence analysis for multidimensional interpolants}

The convergence with respect to the total number of collocation points for the  
tensor product, sparse isotropic, and anisotropic Smolyak approximations
was analyzed in 
\cite{Babuska:2007,Nobile:2008_1,Nobile:2008_2}.  In what follows, 
our goal is 
to prove the bounds on the interpolation error in the approximate solutions $u_{h_k}$ and the functionals $\psi(u_{h_k})$, for ${k \in \mathbb N_0}$. Let
$W$ denote a general Banach space and, for $n \in \{1, \dots, N\}$, define $\Gamma_n^* := \prod_{i \neq n} \Gamma_i$ and let $\bm y_n^*$ denote an arbitrary element of $\Gamma_n^*$. We will use the following theorem, proved in \cite{Nobile:2008_1,Nobile:2008_2}.  

\begin{theorem}
\label{theorem:SCerror}
Let $v\in C^0(\Gamma;W)$ be such that for each random direction $n \in \{1, \dots, N\}$, $v$ admits an analytic extension in the region of the complex plane $\Sigma(n;\tau_n) := \{z\in \mathbb{C}: dist(z,\Gamma_n) \leq \tau_n\}$. Then, there exist constants $C(N)$ and $\mu(N)$, depending on $N$, such that
	\begin{equation*}
	 	\|v - \mathcal{A}^{p,g}_{L} v \| _{L^2_{\rho}(\Gamma;W)} \leq \, C(N) \, 
		M_L^{-\mu(N)}\, \zeta(v)  ,
	\end{equation*}
where $M_L$ is the number of points used by $\mathcal{A}^{p,g}_{L}$ and
	\begin{equation}\label{eq:zeta}
		\zeta(v) \equiv  \max_{1\leq n \leq N} \max_{{\bm y}_n^* \in \Gamma_n^*} \max_{z \in \Sigma(n;\tau_n)} \|v(z,{\bm y}_n^*)\|_{W}.
	\end{equation}
\end{theorem}

\begin{remark}\ {{\bf Anisotropic sparse grid approximations}.}
For anisotropic Smolyak approximations, we use a weight vector $\bm{\alpha} = (\alpha_1,\ldots, \alpha_N)$, where $\alpha_n$ is related to 
$\tau_n$ in Lemma \ref{lem:u_analytic} and corresponds to the rate of convergence 
of the { best approximation} of an analytic function by a polynomial;
see \cite[section 2.2]{Nobile:2008_2}.
In particular, 
for $v\in  C^0(\Gamma_n;W)$ that admits an analytic
  extension in the region of the complex plane $\Sigma(n;\tau_n) =
  \{z\in\mathbb{C}, \;\; \mbox{dist}(z,\Gamma_n)\leq\tau_n\}$ for some
  $\tau_n>0$, we have
\begin{equation}
\label{eq:best}
 \min_{w\in \mathcal{P}_{\ppp(l_n)}} \|v-w\|_{C^0(\Gamma_n;W)}
  \leq \frac{2}{e^{2\alpha_n}-1} e^{- 2 \ppp({l_n}) \alpha_n} \, \max_{z\in\Sigma(\Gamma_n;\tau_n)}\|v(z)\|_{_{W}} 
 \end{equation}
with
\begin{equation*}
\displaystyle{0<\alpha_n=\frac{1}{2}\log\bigg(\frac{2\tau_n}{|\Gamma_n|}+\sqrt{1+\frac{4\tau_n^2}{|\Gamma_n|^2}}\bigg)}.
\end{equation*}
For an {\em isotropic grid}, all the components of the weight vector $\bm{\alpha}$ are the same so that one has to take the worst case scenario, i.e., choose the components of $\bm{\alpha}$ to all equal $\alpha_{{min}}$, because one has no choice but to assume the worst convergence rate in \eqref{eq:best}. 
\end{remark}
\begin{remark}\ {{\bf Dimension-dependent convergence rate}.} 
Table \ref{tbl:SCconv} provides specific values of the convergence rate $\mu(N)$ in Theorem \ref{theorem:SCerror} for some choices of grids, including some sparse grids. In particular, the abscissas are here chosen as the nested Clenshaw-Curtis points, with the mapping $\ppp(l)$ 
given in Table \ref{sgex} describing the corresponding polynomial approximation space. 
As seen in Table \ref{tbl:SCconv}, the asymptotic rate of convergence $\mu$ in general deteriorates with growing dimension $N$ of the stochastic space. The use of sparse grid SC methods is hence only of interest for dimensions $N$ for which $\mu \geq 1/2$ so that the error still converges faster than the corresponding Monte Carlo sampling error. The multilevel approximation presented in this paper suffers from the same deterioration of convergence rate, and roughly speaking, the MLSC method can improve on the multilevel Monte Carlo method only when standard SC performs better that standard Monte Carlo; see \cite[Theorem 4.1]{Cliffe:2011}. 
\end{remark}

\begin{table} [h!]	
	\caption{Convergence rates for $N$-dimensional interpolation operators; see Theorem \ref{theorem:SCerror}.}
	\begin{center}
	\renewcommand{\arraystretch}{1.4}
	\begin{tabular}{ l c}
		\multicolumn{1}{c}{Grid type} & $\mu(N)$ \\ \hline \vspace{0.5ex}
		Full tensor product & $\frac{\alpha_{{min}}}{N}$ \\ \vspace{0.5ex} 
		Classical Smolyak & $\frac{\alpha_{{min}}}{1+\log(2N)}$\\ \vspace{0.5ex}
		Anisotropic classical Smolyak & $\frac{\alpha_{{min}} (\log(2)e - 1/2)}{\log(2) + \sum_{n=1}^N \frac{\alpha_{{min}}}{\sigma_n}}$\\[1ex]
		\hline
	\end{tabular}
	\end{center}
	\label{tbl:SCconv}
\end{table}

\subsection{Multilevel approximation using sparse grids}
\label{sec:sgfun}

The goal of this section is to prove the analyticity assumptions of Theorem \ref{theorem:SCerror} for the approximate solutions $u_{h_k}$ and the functionals $\psi(u_{h_k})$ for 
${k \in \mathbb N_0}$ and hence conclude that the corresponding sparse grid interpolants satisfy assumptions {\bf A5} and {\bf F2}, respectively. 
We start with the required result for the approximate solutions $u_{h_k}$; this result was proved in \cite[Lemmas 3.1 and 3.2]{Babuska:2007}. 

\begin{lemma}\label{lem:u_analytic} If $f \in C^0(\Gamma;L^2(D))$ and if $a \in C^0_{loc}(\Gamma; L^\infty(D))$ and uniformly bounded away from zero, then the solution of the problem \eqref{eq:mod} satisfies $u \in C^0(\Gamma; H^1_0(D))$.
Furthermore, under the assumption that, for every $\bm y=(y_n,\bm y_n^*) \in \Gamma$, there exists $\gamma_n < \infty$ such that for all $j \in \mathbb N_0$
\begin{equation*}
\left\|\frac{\partial^j_{y_n} a(\bm y)}{a(\bm y)}\right\|_{L^\infty(D)} \leq \gamma_n^j\, j! \qquad \text{and} \qquad \frac{\|\partial^j_{y_n} f\|_{L^2(D)}}{1 + \|f\|_{L^2(D)}} \leq \gamma_n^j\, j!,
\end{equation*}
the solution $u(y_n,\bm y_n^*,\mathbf{x})$ as a function of $y_n$, $u: \Gamma_n \rightarrow C^0(\Gamma_n^*; H^1_0(D))$, admits an analytic extension $u(z,\bm y_n^*,\mathbf{x})$, $z \in \mathbb C$, in the region of the complex plane $\Sigma(n;\tau_n)$ with $0 < \tau_n < (2\gamma_n)^{-1}$. Moreover, for all $z \in \Sigma(n;\tau_n)$,
\begin{equation*}
\|u(z)\|_{C^0(\Gamma_n^*; H^1_0(D))} \leq \lambda
\end{equation*}
for some constant $\lambda >0$ independent of $n$.
\end{lemma}

Define the Banach space $\Lambda(\Gamma; H^1_0(D))$ consisting of all functions $ v \in C^0(\Gamma;H^1_0(D))$ such that, for all $ n \in \{1, \dots, N\}$, $v$ admits an analytic extension in the region $\Sigma(n;\tau_n)$. It follows from Lemma \ref{lem:u_analytic} that, under appropriate assumptions on $a$ and $f$, we have $u \in \Lambda(\Gamma; H^1_0(D))$. Because the dependence on $\bm y$ is unchanged in the approximate solution $u_{h_k}$, it also follows that $u_{h_k} \in \Lambda(\Gamma; H^1_0(D))$ for all $k \in \mathbb N_0$, and hence also $u_{h_k} - u_{h_{k-1}} \in \Lambda(\Gamma; H^1_0(D))$ for all $k \in \mathbb N$.

Similar to  {\bf A4}, it follows from standard finite element theory \cite{brenner_scott,ciarlet} that with $W=H^1_0(D)$ and $\zeta$ as in \eqref{eq:zeta}, $\zeta(u_{h_k})$ can be bounded by a constant independent of $k$, whereas $\zeta(u_{h_k} - u_{h_{k-1}})$ can be bounded by a constant multiple of $h_k^\alpha$ for some $\alpha>0$. In general, the constants appearing in these estimates will depend on norms of $a$ and $f$ as well as on the mesh refinement parameter $\eta$. We can hence conclude that with $\mathcal{I}_{M_k} = \mathcal{A}^{p,g}_{L_k}$, assumption {\bf A5} is satisfied for the interpolation schemes considered in Theorem \ref{theorem:SCerror}.  Therefore, for the numerical examples presented in 
Section \ref{sec:num}, we utilize the sparse grid stochastic collocation 
as the interpolatory scheme.

Now we verify the analyticity assumption in Theorem \ref{theorem:SCerror} also for the functionals $\psi(u)$. Because Lemma \ref{lem:u_analytic} already gives an analyticity result for $u$, we use the following result, which can be found in \cite{w65}, about the composition of two functions on general normed vector spaces.

\begin{theorem}\ Let $X_1$, $X_2$, and $X_3$ denote normed vector spaces and let $\theta: X_1 \rightarrow X_2$ and $\nu: X_2 \rightarrow X_3$ be given. Suppose that $\theta$ admits an analytic extension to the set $\tilde X_1$ and $\nu$ admits an analytic extension to the set $\tilde X_2$, with $\theta(\tilde X_1) \subseteq \tilde X_2$.
Then, the composition $\nu \circ \theta : X_1 \rightarrow X_3$ admits an analytic extension to the set $\tilde X_1$.
\end{theorem}

Hence, if we can show that $\psi$ is an analytic function of $u$, we can conclude that $\psi(u)$ is an analytic function of $y_n$. To this end, we need the notion of analyticity for functions defined on general normed vector spaces, which we will now briefly recall.

Given normed vector spaces $X_1$ and $X_2$ and an infinitely Fr\`echet differentiable function $\theta :X_1 \rightarrow X_2$, we can define a Taylor series expansion of $\theta$ at the point $\xi$ in the following way \cite{chae}: 
\begin{equation}\label{eq:taylor}
T_{\theta,\xi}(x) = \sum_{j=0}^\infty \frac{1}{j!} \, d^j \theta(\xi) (x-\xi)^j,
\end{equation}
where $x, \xi \in X_1$, the notation $(x-\xi)^j$ denoting the $j$-tuple $(x-\xi,\dots,x-\xi)$ and $d^j \theta(\xi)$ denoting the $j$-linear operator corresponding to the $j$-th Fr\`echet differential $D^j \theta(\xi)$. The function $\theta$ is then said to be {\em analytic} in a set $Z \subset X_1$ if, for every $z \in Z$, $T_{\theta,z}(x) = \theta(x)$ for all $x$ in a neighbourhood $N_r(z) = \{x \in Z : \|x-z\|_{X_1} < r\}$, for some $r > 0$. 

A sufficient condition for $\theta$ to be analytic in a set $Z$ is thus that $\|d^j \theta(z)\| \leq C^j j!$ for all $z \in Z$ and some $C < \infty$, where $\|\cdot\|$ denotes the usual operator norm. Note that this condition is trivially satisfied if $\|d^j f(z)\| = 0$ for all $z \in Z$ and all $j \geq j^*$, for some $j^* \in \mathbb N$.

\begin{lemma}\label{lem:psi_analytic} Let the assumptions of Lemma \ref{lem:u_analytic} be satisfied. Suppose $\psi$, viewed as a mapping from $C^0(\Gamma_n^*; H^1_0(D))$ to $C^0(\Gamma_n^*; \mathbb R)$, admits an analytic extension to the set $\Sigma(u) \subseteq C^0(\Gamma_n^*; H^1_0(D; \mathbb C))$, and $u(z; y_n^*,x) \in \Sigma(u)$ for all $z \in \Sigma(n;\tau_n)$. Then, $\psi \circ u$, viewed as a mapping from $\Gamma_n$ to $C^0(\Gamma_n^*; \mathbb R)$ admits an anlytic extension to the set $\Sigma(n;\tau_n)$.
\end{lemma} 

Together with Theorem \ref{theorem:SCerror}, now with $W = \mathbb R$, it then follows from Lemma \ref{lem:psi_analytic} that assumption {\bf F2} in Theorem \ref{thm:comp_func} is satisfied for the interpolation schemes considered in this section, provided the functional $\psi$ is an analytic function of $u$. Note that the function $\zeta$ in Theorem \ref{theorem:SCerror} acts on $\psi(u)$ instead of $u$ in this case, leading to optimal convergence rates in $h$ of the stochastic interpolation error.

To finish the analysis, we give some examples of functionals that satisfy the assumptions of Lemma \ref{lem:psi_analytic}.

\begin{example}\ {\em(}Bounded linear functionals{\em)} {\em In this case, there exists a constant $C_\psi$ such that $|\psi(w)| \leq C_\psi \|w\|_{H^1_0(D)}$ for all $w \in H^1_0(D)$ so that $\psi : C^0(\Gamma_n^*; H^1_0(D)) \rightarrow C^0(\Gamma_n^*; \mathbb R)$. Furthermore, for any $v,w \in C^0(\Gamma_n^*; H^1_0(D))$, we have 
\[
d \psi(v)(w_1) = \psi (w_1) \qquad \text{and} \qquad d^j \psi(v) \equiv 0 \quad \forall\, j \geq 2
\]
which implies that $\psi$ admits an analytic extension to all of $C^0(\Gamma_n^*; H^1_0(D; \mathbb C))$.
Examples of bounded linear functionals include point evaluations of the solution $u$ in one spatial dimension and local averages of the solution $u$ in some subdomain $D^* \subset D$, computed as $\frac{1}{|D^*|} \int_{D^*} u \mathrm{d}x$, in any spatial dimension.}
\end{example}

\begin{example}\ {\em(}Higher order moments of bounded linear functionals{\em)} {\em As a generalization of the above example, consider the functional $\psi(v) = \phi(v)^q$, for some bounded linear functional $\phi$ on $H^1_0(D)$ and some $q \in \mathbb N$. As before, it follows from the boundedness of $\phi$ that $\psi : C^0(\Gamma_n^*; H^1_0(D)) \rightarrow C^0(\Gamma_n^*; \mathbb R)$. For any $v \in C^0(\Gamma_n^*; H^1_0(D))$, the differentials of $\psi$ are 
\begin{align*}
&d^j \psi(v)(w_1,\dots,w_j) = \phi(v)^{q-j} \prod_{i=1}^j (q-i+1) \, \phi(w_i), &1 \leq j \leq q, \\
&d^j \psi(v) \equiv 0, & j \geq q+1,
\end{align*}
from which it follows that $\psi$ admits an analytic extension to all of $C^0(\Gamma_n^*; H^1_0(D; \mathbb C))$.}
\end{example}

\begin{example}\ {\em(}Spatial $L^2$-norm{\em)} {\em Consider the functional $\psi(v) = \int_D v^2 \mathrm{d}x = \|v\|_{L^2(D)}^2$. In this case, it follows from the Poincar\`e inequality that $\psi : C^0(\Gamma_n^*; H^1_0(D)) \rightarrow C^0(\Gamma_n^*; \mathbb R)$. For any $v \in C^0(\Gamma_n^*; H^1_0(D))$, the differentials of $\psi$ are 
\[
d\psi(v)(w_1) = 2 \int_D v w_1, \quad d^2\psi(v)(w_1,w_2) = 2 \int_D w_2 w_1, \quad \text{and} \quad d^j \psi(v) \equiv 0\,\,\, \forall\, j \geq 2
\]
which implies that $\psi$ admits an analytic extension to all of $C^0(\Gamma_n^*; H^1_0(D; \mathbb C))$.

For the functional $\psi(v) = \|v\|_{L^2(D)}$, we use the identity $a - b = (a^2-b^2)/(a+b)$ to derive that, for all $v \neq 0$, 
\[
d\psi(v)(w_1) = \frac{\int_D v w_1 \mathrm{d}x}{\|v\|_{L^2(D)}}, \quad d^2\psi(v)(w_1,w_2) = \frac{\int_D w_2 w_1}{\|v\|_{L^2(D)}}, \quad \text{and} \quad d^j \psi(v) \equiv 0 \,\,\, \forall\, j \geq 2.
\]
It then follows that $\psi$ admits an analytic extension to any subset $\Sigma(u) \subseteq C^0(\Gamma_n^*; H^1_0(D; \mathbb C))$ not containing $0$.
The analysis in this example can easily be extended to the functionals $\|v\|_{H^1_0(D)}$ and $\|v\|^2_{H^1_0(D)}$.}
\end{example}

\section{Numerical Examples}\label{sec:num}
{
The aim of this section is to demonstrate numerically the significant reductions in computational cost possible with the use of the MLSC approach.
As an example, consider the following boundary value problem on either $D=(0,1)$ or $D=(0,1)^2$:
\begin{equation}
\label{modnum1}
\left\{
\begin{array}{rll}
-\nabla \cdot \left(a({\bm y}, \mathbf{x}) \nabla u({\bm y}, \mathbf{x})\right) &= 1 \ &\quad \mathrm{for} \ \mathbf{x} \in D \\
u({\bm y},\mathbf{x}) &= 0 &\quad \mathrm{for} \ \mathbf{x} \in \partial D.
\end{array}
\right.
\end{equation}
The coefficient $a$ takes the form
\begin{equation}\label{eq:coeff_num}
a({\bm y}, \mathbf{x}) = 0.5 + \exp\left[ \sum_{n=1}^N \sqrt{\lambda_n} b_n(\mathbf{x}) y_n  \right],
\end{equation}
where $\{y_n\}_{n \in \mathbb N}$} is a sequence of independent, uniformly distributed random variables on [-1,1] and $\{\lambda_n\}_{n \in \mathbb N}$ and $\{b_n\}_{n \in \mathbb N}$ are the eigenvalues and eigenfunctions of the covariance operator with kernel function $C(x,x') = \exp[-\|\mathbf{x}-\mathbf{x}'\|_1]$. Explicit expressions for $\{\lambda_n\}_{n \in \mathbb N}$ and $\{b_n\}_{n \in \mathbb N}$ are computable \cite{Ghanem_Spanos_1991}. In the case $D=(0,1)$, we have
\[
\lambda_n^{1\mathrm D} = \frac{2}{w_n^2+1}\quad\mbox{and}\quad b_n^{1\mathrm D}(\mathbf{x}) = A_n (\sin (w_n \mathbf{x}) + w_n \cos (w_n \mathbf{x})) \quad \text{for all } n \in \mathbb N, 
\]
where $\{w_n\}_{n \in \mathbb N}$ are the (real) solutions of the transcendental equation
\[
\tan(w) = \frac{2\,w}{w^2-1}
\]
and the constant $A_n$ is chosen so that $\|b_n\|_{L^2(0,1)}=1$. In two spatial dimensions, with $D=(0,1)^2$, the eigenpairs can be expressed as
\[
\lambda_n^{2\mathrm D} = \lambda_{i_n}^{1\mathrm D} \, \lambda_{j_n}^{1\mathrm D}\quad\mbox{and}\quad b_n^{2\mathrm D} = b_{i_n}^{1\mathrm D} \, b_{j_n}^{1\mathrm D}
\]
for some $i_n, j_n \in \mathbb N$. In both one and two spatial dimensions, the eigenvalues $\lambda_n$ decay quadratically with respect to $n$  \cite{charrier12}.

It is shown in \cite[Example 3]{Babuska:2007} that with the coefficient $a$ of the form \eqref{eq:coeff_num} and $f = 1$, the assumptions of Lemma \ref{lem:u_analytic} are satisfied with $\gamma_n = \sqrt{\lambda_n} \|b_n\|_{L^\infty(D)}$.

For spatial discretization, we use continuous, piecewise-linear finite elements on uniform triangulations of $D$, starting with a mesh width of $h=1/2$. As interpolation operators, we choose the (isotropic) sparse grid interpolation operator
\eqref{eq:SG}, using $p$ and $g$ given by the classic Smolyak approximation in Table \ref{sgex},
based on Clenshaw-Curtis abscissas; see Section \ref{sec:SC}. 

The goal of the computations is to estimate the error in the expected value of a functional $\psi$ of the solution of \eqref{modnum1}. For fair comparisons, all values of $\varepsilon$ reported are relative accuracies, i.e., we have scaled the errors by the value of $\mathbb E[\psi(u)]$ itself. We consider two different settings: in Section \ref{ssec:num1}, we consider problem \eqref{modnum1} in two spatial dimensions with $N=10$ random variables whereas, in Sections \ref{ssec:num2} and \ref{ssec:num3}, we work in one spatial dimension with $N=20$ random variables. Because the exact solution $u$ is unavailable, the error in the expected value of $\psi(u)$ has to be estimated. In Sections \ref{ssec:num1} and \ref{ssec:num2}, we compute the error with respect to an ``overkilled'' reference solution obtained using a fine mesh spacing $h^*$ and high interpolation level $L^*$. However, because this is generally not feasible in practice, we show in Section \ref{ssec:num3} how the error can be estimated when the exact solution is not available and one cannot compute using a fine spatial mesh and high stochastic interpolation level. The cost of the multilevel estimators is computed as discussed in Section \ref{ssec:comp_analysis} and Remark \ref{rem:canc}, with $\gamma=d$, i.e., by assuming the availability of an optimal linear solver. For non-optimal linear solvers for which $\gamma > d$, the savings possible with the multilevel approach will be even greater than demonstrated below.

\subsection{$\mathbf{d=2, N=10}$}\label{ssec:num1}

As the quantity of interest, we choose the average value of $u$ in a neighborhood of the midpoint $(1/2,1/2)$, computed as $\psi(u) = \frac{1}{|D^*|} \int_{D^*} u(\mathbf{x}) \mathrm{d}\mathbf{x}$, where $D^*$ denotes the union of the six elements adjacent to the node located at $(1/2,1/2)$ of the uniform triangular mesh with mesh size $h=1/256$.

We start by confirming, in Figure \ref{fig:2d_1}, the assumptions of Theorem \ref{thm:comp_func}. The reference values are computed with spatial mesh width $h^*=1/256$ and stochastic interpolation level $L^*=5$. 

\begin{figure}[h!]
\centering
\includegraphics[width=0.45\textwidth]{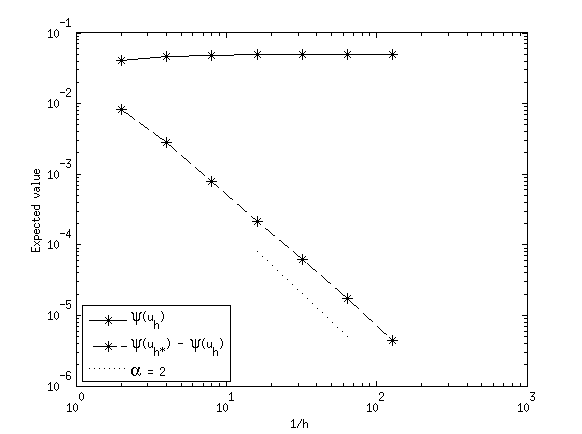}\includegraphics[width=0.45\textwidth]{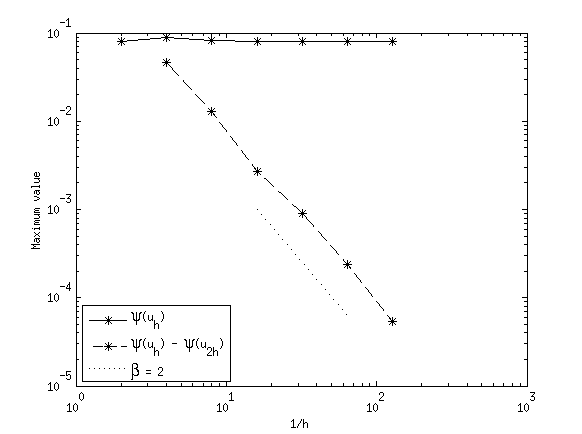}
\includegraphics[width=0.45\textwidth]{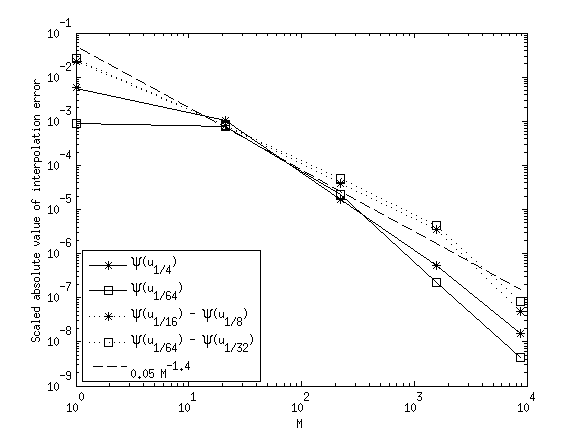}\includegraphics[width=0.45\textwidth]{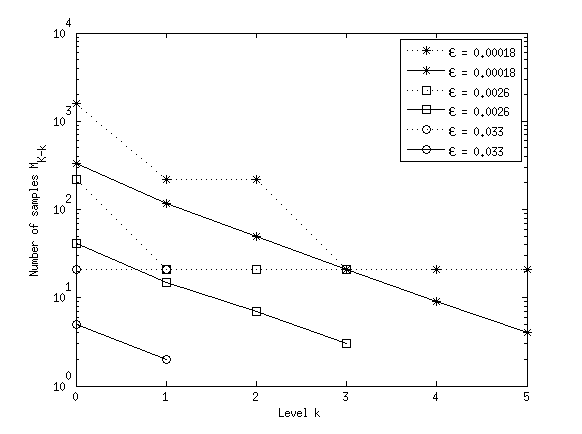}
\caption{$D=(0,1)^2$, $N=10$. Top left: $\mathbb E[\mathcal I_5 \psi(u_h)]$ and $\mathbb E[\mathcal I_5 \psi(u_{1/256}) - \mathcal I_5 \psi(u_h)]$ versus $1/h$. Top right: $\zeta(\psi(u_h))$ and $\zeta(\psi(u_h) - \psi(u_{2h}))$ versus $1/h$. Bottom left: $\mathbb E[\mathcal I_5 \psi(u_h) - \mathcal I_l \psi(u_h)]/h_0^2$ and $[\mathcal I_5 (\psi(u_h) - \psi(u_{2h}))- \mathcal I_l (\psi(u_h)- \psi(u_{2h})]/h^2$ versus $M_l$, for various $h$. Bottom right: number of samples $M_{K-k}$ versus $k$.}
\label{fig:2d_1}
\end{figure}

The top-left plot of Figure \ref{fig:2d_1} shows the convergence of the finite element error in the expected value of $\psi(u)$, and confirms that assumption {\bf F1} of Theorem \ref{thm:comp_func} holds with $\alpha = 2$. 

The top-right plot of Figure \ref{fig:2d_1} shows the behavior of the function $\zeta$ from Theorem \ref{theorem:SCerror} for the quantities $\psi(u_h)$ and $\psi(u_h) - \psi(u_{2h})$. We see that whereas $\zeta(\psi(u_h))$ is bounded by a constant independent of $h$ and $\zeta(\psi(u_h) - \psi(u_{2h}))$ decays quadratically in $h$. This confirms assumption {\bf F3} with $\beta=2$.  

The bottom-left plot of Figure \ref{fig:2d_1} shows the interpolation error in $\psi(u_{h})$ scaled by $h_0^2$ and the interpolation error in $\psi(u_{h}) - \psi(u_{2h})$ scaled by $h^2$ for several  values of $h$. According to assumptions {\bf F2} and {\bf F3}, these plots should all result in a straight line $C M^{-\mu}$, where $C=C_I C_\zeta$. The best fit which has $C=0.05$ and $\mu=1.4$ is added for comparison. 

The bottom-right plot of Figure \ref{fig:2d_1} shows the number of samples $M_k$ computed using the formula \eqref{eq:opt_eta}, with $C=0.05$ and $\mu=1.4$, for several values of $\varepsilon$. The finest level $K$ was determined using the estimates on the finite element error from the top-left plot. Solid lines correspond to numbers rounded up to the nearest integer, as is done in \eqref{eq:opt_etar}, whereas dotted lines correspond to the number of samples rounded up to the next level of the sparse grid. 

In Figure \ref{fig:2d_2}, we study the cost of the standard and multilevel collocation methods to achieve a given total accuracy $\varepsilon$. In both plots, the data labeled `SC' and `MLSC' denote standard and multilevel stochastic collocation, respectively. For data labeled `formula', the number of samples was determined by the formula \eqref{eq:opt_eta} with $C=0.05$ and $\mu=1.4$, rounded up to the next sparse grid level (the dotted lines in the bottom right plot of Figure \ref{fig:2d_1}). For data labeled `best', the number of samples was chosen manually so as to achieve a total accuracy $\varepsilon$ for the smallest computational cost. For all methods, we chose $h_0=1/4$.

\begin{figure}[h!]
\centering
\includegraphics[width=0.45\textwidth]{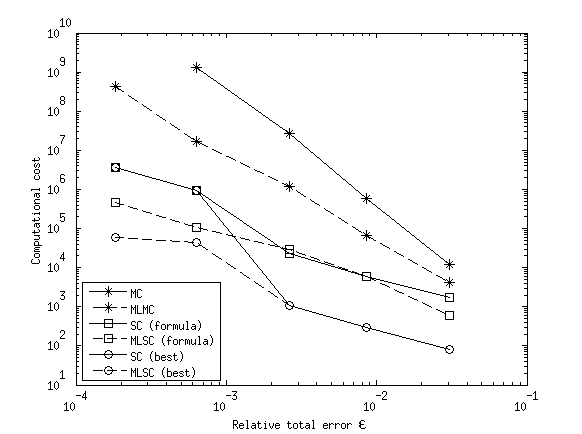}\includegraphics[width=0.45\textwidth]{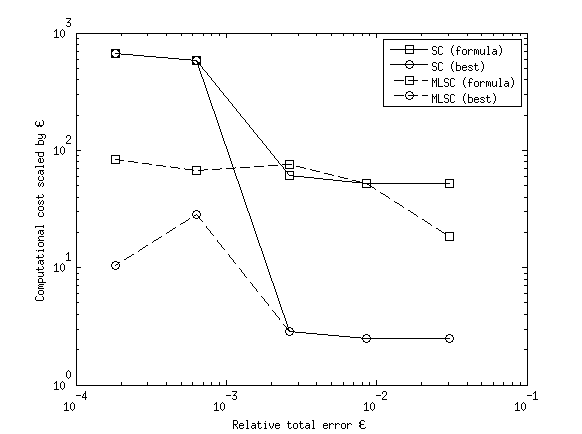}
\caption{$D=(0,1)^2$, $N=10$. Left: computational cost versus relative error $\varepsilon$. Right: computational cost scaled by $\varepsilon^{-1.36} $ versus relative error $\varepsilon$.}
\label{fig:2d_2}
\end{figure}

In the left plot of Figure \ref{fig:2d_2}, we simply plot the computational cost of the different estimators against $\varepsilon$. For comparison, we have also added corresponding results for Monte Carlo (MC) and multilevel Monte Carlo (MLMC) estimators. In both the `formula' and the `best' case, the multilevel collocation method outperforms standard SC. Both collocation-based methods outperform both Monte Carlo approaches.

In the right plot in Figure \ref{fig:2d_2}, we compare the observed computational cost with that predicted by Theorem \ref{thm:comp_func} for the standard and multilevel collocation methods. In our computations, we observed $\alpha \approx 2$, $\beta \approx 2$, and $\mu \approx 1.4$, which with $\gamma=2$ gives computational costs of $\varepsilon^{-1}$ and $\varepsilon^{-1.72}$ for the multilevel and standard SC method, respectively. We therefore plot the computational cost scaled by $\varepsilon^{1}$. We see that both multilevel methods indeed seem to grow approximately like $\varepsilon^{-1}$, with the `formula' case growing slightly faster for large value of $\varepsilon$ and the `best' case growing slightly faster for small values of $\varepsilon$. The costs for both standard collocation methods grow a lot faster with $\varepsilon$.

Figure \ref{fig:2d_3} provides results for a different quantity of interest, $\psi(u) = \|u\|_{L^2(D)}$. The left plot corresponds to the bottom-left plot in Figure \ref{fig:2d_1} and again confirms that the interpolation error in $\psi(u_{h}) - \psi(u_{2h})$ scales with $h^2$. The right plot corresponds to the left plot of Figure \ref{fig:2d_2}, where we plot the computational cost of the different estimators against $\varepsilon$. We see that all collocation-based methods outperform the Monte Carlo approaches. In both the `formula' and the `best' case, the multilevel collocation method again outperforms standard SC.

\begin{figure}[h!]
\centering
\includegraphics[width=0.45\textwidth]{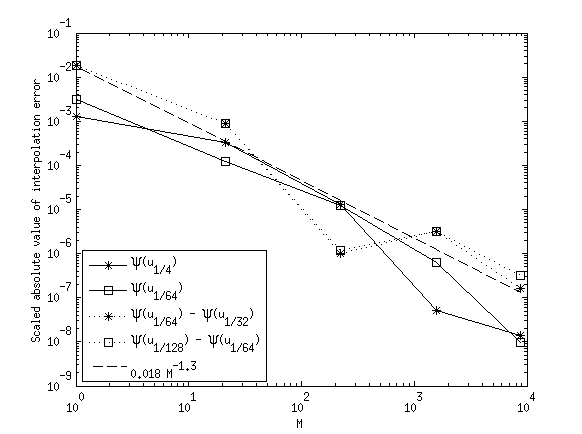}\includegraphics[width=0.45\textwidth]{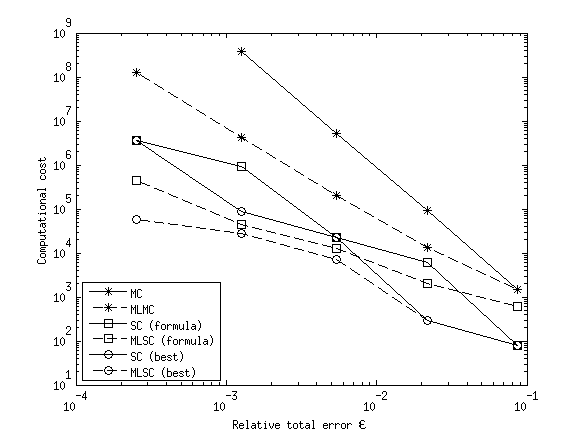}
\caption{$D=(0,1)^2$, $N=10$. Left: $\mathbb E[\mathcal I_5 \psi(u_h) - \mathcal I_{M_k} \psi(u_h)]/h_0^2$ and $[\mathcal I_5 (\psi(u_h) - \psi(u_{2h}))- \mathcal I_{M_k}(\psi(u_h)- \psi(u_{2h})]/h^2$ versus $M_k$, for various $h$. Right: computational cost versus relative error $\varepsilon$.}
\label{fig:2d_3}
\end{figure} 

\begin{remark}\
Before considering the second model problem, let us briefly comment on the differences between the `best' and the `formula' multilevel methods. The `formula' multilevel collocation method performs sub-optimally mainly for two reasons. First, it always rounds up the number of samples $M_k$ to the nearest sparse grid level, which may be substantially higher than the number of samples actually required. Secondly, it does not take into account sign changes in the interpolation error, which in practice can lead to significant reductions in the interpolation error of the multilevel method. For both of these reasons, the interpolation error is often a lot smaller than the required $\varepsilon/2$, leading to sub-optimal performance. This issue is partly addressed in Section \ref{ssec:num3}, where we consider not always rounding up, but rounding the number of samples either up or down to the nearest sparse grid level.
\end{remark}

\subsection{$\mathbf{d=1, N=20}$}\label{ssec:num2}

We now repeat the numerical tests done in the previous section for the case $D=(0,1)$ and $N=20$. For the quantity of interest, we choose the expected value of the solution $u$ evaluated at $x=\frac34$. The reference values are computed using the mesh width $h^*=1/1024$ and interpolation level $L^*=5$. 

We again start by confirming, in Figure \ref{fig:1d_1}, the assumptions of Theorem \ref{thm:comp_func}. The four plots of that figure convey the same information as do the corresponding plots in Figure \ref{fig:2d_1} and again confirm assumptions {\bf F1}, {\bf F2}, and {\bf F3} of that theorem with $\alpha = 2$ and $\beta=2$ and, in the bottom-right plot, the best line fit $C=C_I C_\zeta$ with $C=0.005$ and $\mu=0.8$.

\begin{figure}[h!]
\centering
\includegraphics[width=0.45\textwidth]{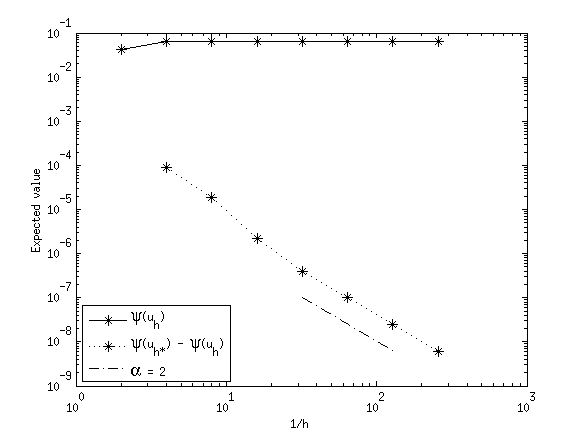}\includegraphics[width=0.45\textwidth]{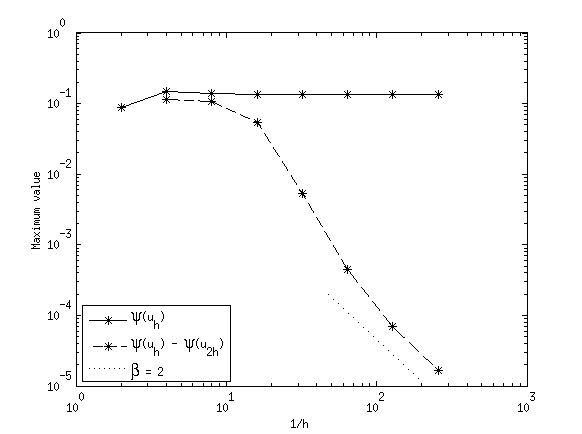}
\includegraphics[width=0.45\textwidth]{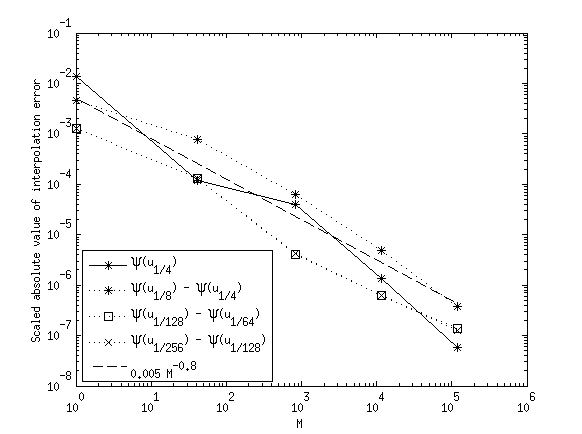}\includegraphics[width=0.45\textwidth]{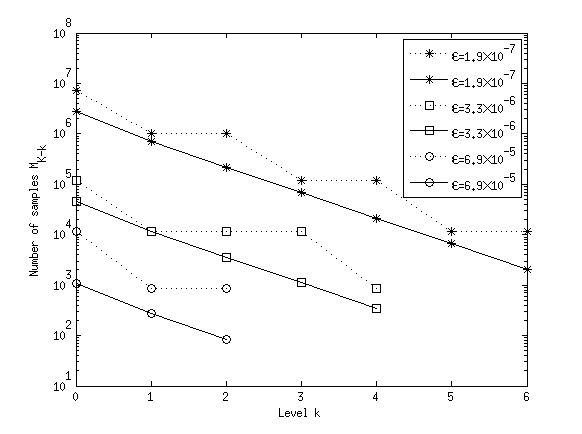}
\caption{$D=(0,1)$, $N=20$. Top left: $\mathbb E[\mathcal I_5 \psi(u_h)]$ and $\mathbb E[\mathcal I_5 \psi(u_{1/256}) - \mathcal I_5 \psi(u_h)]$ versus $1/h$. Top right: $\zeta(u_h)$ and $\zeta(u_h - u_{2h})$ versus $1/h$. Bottom left: $\mathbb E[\mathcal I_5 \psi(v) - \mathcal I_{M_k} \psi(v)]/\zeta(v)$ versus $M_k$ for various $v$. Bottom right: number of samples $M_{K-k}$ versus $k$.}
\label{fig:1d_1}
\end{figure}

Figure \ref{fig:1d_2} conveys the same information and uses the same labeling as does Figure \ref{fig:2d_2}. Again, for both the `formula' and `best' cases, the multilevel collocation method eventually outperforms standard SC and both collocation-based methods also outperform the Monte Carlo approaches. Based on the values $\alpha \approx 2$, $\beta \approx 2$, and $\mu \approx 0.8$, Theorem \ref{thm:comp_func} now predicts the computational costs of $\varepsilon^{-1.25}$ and $\varepsilon^{-1.75}$ for the multilevel and the standard collocation methods, respectively. The right-plot in Figure \ref{fig:1d_2} indicates that the `formula' multilevel collocation method indeed seems to grow like $\varepsilon^{-1.25}$ whereas the `best' multilevel method actually seems to grow slower for small values of $\varepsilon$.  This is likely due to the different signs of the interpolation errors in the multilevel estimator. Also, again, the costs for both standard collocation methods grow a lot faster with $\varepsilon$.

\begin{figure}[h!]
\centering
\includegraphics[width=0.45\textwidth]{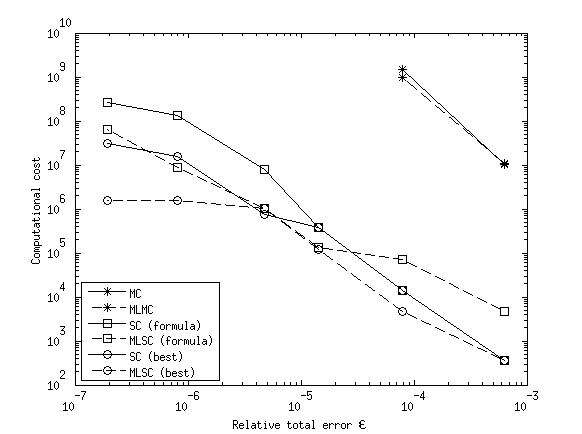}\includegraphics[width=0.45\textwidth]{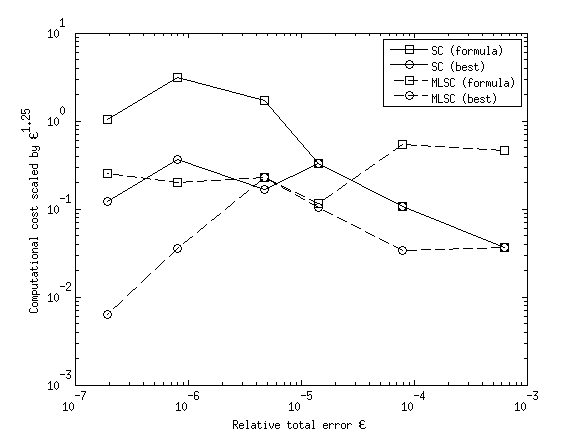}
\caption{$D=(0,1)$, $N=20$. Left: computational cost versus relative error $\varepsilon$. Right: computational cost scaled by $\varepsilon^{-1.25} $ versus relative error $\varepsilon$.}
\label{fig:1d_2}
\end{figure}

\subsection{Practical implementation}\label{ssec:num3}

In Sections \ref{ssec:num1} and \ref{ssec:num2}, the accuracy of the computed estimates was assessed by comparison to a reference solution. Of course, in practice, a fine-grid, high-level reference solution is not available. Therefore, in this section, we describe how to implement the MLSC method without having recourse to a reference solution. We suggest the following practical strategy that is  similar to the one proposed in \cite{Giles:2008}.

\begin{itemize}
\item[1.] Estimate the constants $\alpha$, $\beta$, $\mu$, and $C = C_I \, C_\zeta$.
\item[2.] Start with $K=1$.
\item[3.] Calculate the optimal number of samples $M_k, k=0,\dots,K$, according to the formula \eqref{eq:opt_eta}, and round to the nearest sparse grid level. 
\item[4.] Test for convergence using
\[
\mathbb E[\psi(u_{h_k}) - \psi(u_{h_{k-1}})] = (\eta^\alpha -1) \, \mathbb E[(\psi(u)-\psi(u_{h_k}))].
\]
\item[5.] If not converged, set $K=K+1$ and return to step 3. 
\end{itemize}

\noindent Note that in this procedure, steps 3 and 4 ensure that the interpolation error and the spatial discretization error are each less than the required tolerance $\varepsilon/2$, respectively.

The estimation of the constants $\alpha$, $\beta$, $\mu$, and $C$ in step 1 can be done relatively cheaply from computations done using mesh widths $h_0$, $h_1$, and $h_2$ and interpolation levels $k=0,1,2$. For the results provided below, we estimated the convergence rate $\alpha$ from the level 1 interpolants $\mathcal I_{1}$ of $\psi(u_0)$, $\psi(u_1)$, and $\psi(u_2)$, resulting in $\alpha\approx2.1$. In light of the results in Section \ref{sec:SC}, we assumed $\beta=\alpha$. We then used the first three interpolation levels of $\psi(u_0)$ and $\psi(u_1)-\psi(u_0)$ to obtain the estimates $C\approx0.01$ and $\mu\approx0.8$. Note that the value of $\mu$ is the same as in Section \ref{ssec:num2} whereas the value of the constant $C$ is slightly larger. This is due to the fact that, for the large values of $h$ used to estimate this constant, the function $\zeta(\psi(u_h) - \psi(u_{2h}))$ has probably not yet settled into its asymptotic quadratic decay.

As mentioned in Section \ref{ssec:num1}, always rounding the number of samples resulting from formula \eqref{eq:opt_eta} up to the next sparse grid level may lead to a substantial increase in the computational cost and hence a sub-optimal performance of the multilevel method. In practice, one might therefore consider not always rounding up, but instead rounding either up or down. As long as we do not round down more frequently than we round up, or at least not much more often, this approach should still result in an interpolation error below the required tolerance $\varepsilon/2$.

Table \ref{tbl:1} shows the number of samples $M_{K-k}$ resulting from the implementation described in this section for the model problem with $d=1$ and $N=20$ from Section \ref{ssec:num2}. For each value of $\varepsilon$, the first row, denoted by `formula', corresponds to the numbers $M_{K-k}$ resulting from formula \eqref{eq:opt_eta} rounded up to the nearest integer. The second row, denoted `up', are the numbers in the first row rounded up to the next corresponding sparse grid level. For the final row, denoted `up/down', the rounding of the number of samples was done in the following way: First, all numbers were rounded either up or down to the nearest corresponding sparse grid level. If this resulted in more numbers being rounded down than up, we chose the number that was rounded down by the largest amount and then instead rounded this number up. This procedure was continued iteratively. The same was done when more numbers were rounded up than down.

\begin{table} [h!]	
\begin{center}
\renewcommand{\arraystretch}{1.25}
\begin{tabular}{ |c|| c |c c c c c  |} \hline
 $\varepsilon$ & level & 0& 1& 2& 3& 4\\ \hline\hline
\multirow{3}{*}{6.3e-4} & formula & 191& 48& 15& &   \\
		      & up & 841& 841& 41& &   \\
		      & up/down & 841& 41& 41& &  \\ \hline
\multirow{3}{*}{7.9e-5} & formula & 3002& 747& 233& 73&   \\
		      & up & 11561& 841& 841& 841&   \\
		      & up/down & 841& 841& 841& 41&   \\ \hline
\multirow{3}{*}{1.4e-5} & formula & 27940& 6949& 2169& 677& 212  \\
		      & up & 120401& 11561& 11561& 841& 841  \\
		      & up/down & 11561& 11561& 841& 841& 841  \\ \hline
\multirow{3}{*}{4.7e-6} & formula & 110310& 27433& 8562& 2672& 834 \\
		      & up & 120401& 120401& 11561& 11561& 841 \\
		      & up/down & 120401& 11561& 11561& 11561& 841 \\ \hline
\end{tabular}
\end{center}
\label{tbl:1}
\caption{$D=(0,1)$, $N=20$. Number of samples $M_{K-k}$ computed using formula \eqref{eq:opt_eta} and various rounding schemes.}
\end{table}

To confirm that the adaptive procedure still achieves the required tolerance on the total error, we have, for Table \ref{tbl:2}, computed the stochastic interpolation and finite element errors (with respect to a reference solution) and the computational cost of the multilevel approximations from Table \ref{tbl:1}. For comparison, we have added the results for the multilevel method which was manually found to give a total error less than $\varepsilon$ at minimal cost, which was already computed in Section \ref{ssec:num2} assuming a reference solutions was available. Note that for large values of $\varepsilon$, the adaptive procedure described in this section overestimated the finite element error, leading to a larger number of levels $K$ compared to that found in Section \ref{ssec:num2}. It is clear from Table \ref{tbl:2} that not only does the alternative rounding procedure yield the required bound on the error, it also significantly reduces the computational cost of the multilevel method, bringing it close to what was manually found to be the minimal cost possible.

\begin{table} [h!]	
\begin{center}
\renewcommand{\arraystretch}{1.25}
\begin{tabular}{ |c| c |c c c|} \hline
 $\varepsilon$ & & Interpolation error& Spatial error &Cost \\ \hline\hline
\multirow{2}{*}{6.3e-4}  & up &  6.7e{-5}& 3.4e{-5}& 8266\\
		      & up/down & 2.8e{-4}& 3.4e{-5}& 4902\\ 
		& best &  8.0e{-5}&  2.9e{-4}&  369\\\hline
\multirow{2}{*}{7.9e-5}   & up &  2.2e{-5}&  6.3e{-6}& 85558\\
		      & up/down &  3.0e{-5}&  6.3e{-6}& 15650\\ 
		& best &  2.4e{-5}&  3.4e{-5}& 4591\\\hline
\multirow{2}{*}{1.4e-5}     & up &  2.7e{-6}&  1.6e{-6}& 853207\\
		      & up/down & 8.3e{-6}&  1.6e{-6}&158714\\ 
		& best &  3.9e{-6}&  6.3e{-6}& 119699\\\hline
\multirow{2}{*}{4.7e-6}   & up & 7.3e{-8}& 1.6e{-6} &1519787\\
		      & up/down &1.2e{-6}&  1.6e{-6} &1038183\\ 
		& best &  1.2e{-6}&  1.6e{-6} &1038183\\\hline
\end{tabular}
\end{center}
\label{tbl:2}
\caption{$D=(0,1)$, $N=20$. Stochastic interpolation and spatial errors (with respect to the reference solution) and computational cost of various multilevel methods.}
\end{table}

\section{Concluding remarks}

Computing solutions of stochastic partial differential equations using stochastic collocation methods can become prohibitively expensive as the dimension of the random parameter space increases. Drawing inspiration from recent work in multilevel Monte Carlo methods, this work proposed a multilevel stochastic collocation method, based on a hierarchy of spatial and stochastic approximations. A detailed computational cost analysis showed, in all cases, a sufficient improvement in costs compared to single-level methods. Furthermore, this work provided a framework for the analysis of a multilevel version of any method for SPDEs in which the spatial and stochastic degrees of freedom are decoupled.

The numerical results practically demonstrated this significant decrease in complexity versus single level methods for each of the problems considered. Likewise, the results for the model problem showed multilevel SC to be superior to multilevel MC even up to $N=20$ dimensions. 

One of the largest obstacles to the practicality of stochastic collocation methods is the huge growth in the number of points between grid levels. In the multilevel case, this can lead to a large amount of computational inefficiency. Certain simple rounding schemes were proposed to mitigate this effect, and proved to be extremely effective for the problems considered. Similarly, since most of our example problems involved computation of a reference solution for the estimation of the necessary constants, a more practical multilevel stochastic collocation algorithm that dispensed with the need for a reference solution was proposed and tested.

It is clear that for any sampling method for SPDEs, whether Monte Carlo or stochastic collocation,  multilevel methods are to be preferred over single-level methods for improved efficiency. Especially in the case of stochastic collocation methods, multilevel approaches enable one to further delay the curse of dimensionality, tempering the explosion of computational effort that results when the stochastic dimension increases. Though Monte Carlo methods are often preferable for problems involving a large stochastic dimension, multilevel approaches greatly improve the effectiveness of stochastic collocation methods versus Monte Carlo methods.

\section*{Acknowledgements}
The first and the fourth authors are 
supported by the US Department of Defense Air Force Office of Scientific Research (AFOSR) under grant number FA9550-11-1-0149 and by he US Department of Energy Advance Simulation Computing Research (ASCR) program under grant number DE-SC0010678.

The second and third authors are supported by the US AFOSR under grant number 1854-V521-12 and by the US Department of Energy ASCR under grant number ERKJ259.  Also supported by the Laboratory Directed Research and Development (LDRD) Program at the Oak Ridge National Laboratory (ORNL).
The ORNL is operated by UT-Battelle, LLC, for the US Department of
Energy under Contract DE-AC05-00OR22725.

\bibliographystyle{siam}
\bibliography{MLSCbib}{}

\end{document}